\documentclass[preprint,sort&comp]{elsarticle}

\usepackage{lineno,hyperref}
\modulolinenumbers[5]

\usepackage{bbm,dsfont}
\usepackage[mathscr]{eucal}
\usepackage[all]{xy} 
\usepackage{graphicx}
\usepackage{caption}
\usepackage{subcaption}
\usepackage{multicol}
 \usepackage{multirow}
\usepackage{epstopdf}
\usepackage{graphics}                           
\usepackage{amssymb,amsfonts,amsthm}    
\usepackage{color}                              
\usepackage[centertags]{amsmath}
\usepackage{amsfonts,amssymb,amsthm,newlfont}
\usepackage[margin=1in]{geometry}
\usepackage{algorithmic,algorithm}
\usepackage{hyperref}
\usepackage{mathrsfs}
\usepackage{enumerate}

\usepackage{nomencl}
\usepackage{fourier}

\theoremstyle{plain}
\newtheorem{thm}{Theorem}[section]

\newtheorem{lemma}[thm]{Lemma}
\newtheorem{corollary}[thm]{Corollary}

\newtheorem{remark}[thm]{Remark}

\numberwithin{equation}{section}
\numberwithin{algorithm}{section}

\def\bv{\mathbf{v}}

 \usepackage{setspace} 
\begin{document}

\begin{frontmatter}

\title{Bilinear Control of Convection-Cooling: From Open-Loop to Closed-Loop 
\tnoteref{mytitlenote}}
\tnotetext[mytitlenote]{Acknowledgment }

\author{Weiwei Hu\corref{cor1}}
\address{Department of Mathematics, University of Georgia, Athens, GA 30602, USA }
\cortext[cor1]{Corresponding author}
\ead{Weiwei.Hu@uga.edu}
\author{Jun Liu\corref{}}
\ead{juliu@siue.edu}
\address{Department of Mathematics and Statistics, Southern Illinois University Edwardsville,
 Edwardsville, IL 62025, USA}
\author{Zhu Wang\corref{}}
\ead{wangzhu@math.sc.edu}
\address{Department of Mathematics, University of South Carolina, Columbia, SC 29208, USA}



\begin{abstract}
This paper is concerned with a bilinear  control problem for enhancing convection-cooling  via an incompressible velocity field.  Both optimal open-loop control and closed-loop feedback control designs are addressed.  First and second order optimality conditions  for characterizing the optimal solution are discussed.    In particular, the method of instantaneous control is applied to establish the feedback laws. 
Moreover, the construction of feedback laws is also  investigated by directly utilizing the optimality system  with appropriate  numerical  discretization schemes.  
Computationally, it is much easier to implement  the closed-loop feedback control than the optimal open-loop  control, as the latter requires to solve the state equations  forward in time, coupled with  the adjoint equations  backward in time  together with a nonlinear optimality condition. 
Rigorous analysis and
numerical experiments are presented  to demonstrate our ideas and  validate the  efficacy of the control designs.
\end{abstract}


\end{frontmatter}


\section{Introduction}

	 The question of the  influence  of advection on diffusion is  a topic of fundamental interest in engineering and natural sciences with broad applications ranging from heat transfer, chemical mixing   on small and large scales, to preventing the spreading of pollutants in geophysical flows. 
Convection-cooling is the mechanism where heat is transferred from a hot object into the ambient air or liquid. %
In general, there are two types of convectional cooling:  natural convection cooling and the forced air convection cooling (cf.~\cite{bergman2011introduction,  bejan2013convection, kreith2012principles}).
The latter is used in designs where the enclosures or environment do not offer an effective natural cooling performance. 
In this work, we are aiming at understanding 
what flows are efficient  in enhancing cooling or the homogenization process and whether it is possible to construct such flows by utilizing  the information of the temperature only.  Specifically,  we are interested in  the  control designs  for convection-cooling via  incompressible fluid flows. To this end, we consider the diffusion-convection model 
for a cooling application in an open bounded and connected domain  $\Omega\subset \mathbb{R}^d, d=2,3$,  with a Lipschitz   boundary  $\Gamma$.  The system of equations reads
  \begin{align}
	\frac{\partial T}{\partial t}&=\kappa \Delta T-\bv \cdot \nabla T\quad \text{in}\quad \Omega, \label{sta_T}\\
	\nabla \cdot \bv&=0 \quad \text{in}\quad \Omega, \label{sta_v}
\end{align}
where  $T$ is the  temperature,  $\kappa>0$  is  the thermal diffusivity, and $\bv$ is a divergence free vector field.
Neumann boundary condition for temperature and  no-slip boundary condition for  velocity are considered, i.e., 
 \begin{align}
 \frac{\partial T }{\partial n}\Big|_{\Gamma} =0\quad \text{and}\quad   \bv |_{\Gamma}=0. \label{BC}
\end{align}	
The initial condition is given by 
 \begin{align}
T(x, 0)=T_0(x). \label{IC}
\end{align}

The diffusion-convection model  \eqref{sta_T}--\eqref{IC} is one of the most studied PDEs in both mathematical and physical literature. Of special note is that the flow velocity  will be taken as the control input in this work. This  naturally leads to a bilinear control problem.   In particular, we like to understand what is the optimal flow velocity that accelerates the convergence of the temperature to its average, and construct such velocity in a feedback form.  Constantin {\it et al.} in \cite{constantin2008diffusion} provided a sharp  characterization of incompressible flows
that  produce a significantly stronger dissipative effect than dissipation alone. 
However, constructing an optimal  velocity field in a feedback form is non-trivial. One of the well-known approaches is to solve the related Hamilton-Jacobi-Bellam (HJB)  differential equation, yet  it suffers the curse of dimensionality.  In this work,  we are aiming at investigating a feasible  nonlinear  feedback control law  for convection-cooling based on  the instantaneous control design  and establish the corresponding  stabilization  results.  The fundamental idea of instantaneous control   is built upon an optimal control problem, which essentially gives rise to a sub-optimal feedback law. Moreover, we also investigate the construction of  feedback laws directly  utilizing  the  discretized optimality conditions. As a first step to implement the feedback control  design,  we start with an optimal control problem 
 seeking  for a velocity  field  that minimizes   the variance of the temperature distribution. The problem can be formulated as follows:   find an incompressible velocity  $\bv$ that   minimizes
 \begin{align*}
J(\bv)&=\frac{\alpha}{2}\|T(x, t_f)-\langle T(x,t_f)\rangle\|^2_{L^2}+\frac{\beta}{2}\int^{t_f}_0\|T-\langle T\rangle\|^2_{L^2}\, dt
 +\frac{\gamma}{2}\|\bv\|^2_{U_{\text{ad}}} \tag{P} \label{eq:p}
 \end{align*}
for a given $t_f>0$,  subject to \eqref{sta_T}--\eqref{IC}, where 
$\langle T\rangle=\frac{1}{|\Omega|} \int_{\Omega} T\, dx$
stands for the spatial average of temperature,
 $\alpha, \beta \geq 0$ and $ \gamma>0$ are the state and control weight parameters, respectively,   and $U_{\text{ad}}$  stands for the set of admissible control.  The parameters   $\alpha$ and  $\beta$ do  not vanish simultaneously.

   For the convenience  of our discussion,  we first introduce the following spaces
\begin{align*} 
H&=\{\bv \in L^2(\Omega)\colon \nabla \cdot \bv=0, \ \bv\cdot n|_{\Gamma}=0\},\qquad
V =\{\bv \in H^1_0(\Omega)\colon \nabla \cdot \bv=0\}.
\end{align*}
The most relevant work on optimal control of the scalar field via incompressible fluid flows  can be found in (cf.~\cite{ liu2008mixing,  barbu2016optimal, hu2017enhancement,  hu2018optimal,  hu2018boundary, hu2018boundarySICON, hu2018approximating, hu2019approximating, glowinski2021bilinear, he2021optimal}),  with applications to heat transfer,  fluid  mixing and optical  flow control problems.
 Due to the advection term $\bv\cdot \nabla T$, the control-to-state map  $\bv\mapsto T$ is bilinear, and  hence problem \eqref{eq:p} becomes  non-convex and the optimal solution may not be unique in general. The choice  of $U_{\text{ad}}$   plays a key role in proving the existence of an optimal solution and deriving the optimality conditions.  
 Establishing the  existence of an optimal velocity field  will involve  a  compactness argument associated with the control-to-state map.   To obtain a steady flow, Liu in \cite{liu2008mixing} penalized the magnitude of the time derivative of $\bv$ in the cost functional, however, this resulted in a nonlinear wave type of optimality conditions, which are difficult  to implement numerically. 
Barbu and Marinoschi in \cite{barbu2016optimal} 
showed the existence of an optimal solution for $\bv\in L^2(0, t_f; H)$,
yet the challenge was encountered in deriving the first order optimality conditions. For $\bv\in L^2(0, t_f; H)$,  it  is not smooth enough to allow the differentiability of the state equations. 
Consequently, the variational inequality or the Euler-Lagrange method   can not be directly applied. Instead, an approximating control approach was employed in   \cite{barbu2016optimal}, which first considered the  velocity in a much smoother space and then showed the convergence of the optimality conditions  for  the approximating control problem to the original one.  Moreover, as shown in  \cite[Theorem 6]{barbu2016optimal}, if further assume that  $\bv\in L^\infty(0, t_f; L^{\infty}(\Omega)\cap H)$, then the uniqueness of the optimal controller can be   obtained by showing the uniqueness of the optimality system under certain conditions. Similar ideas have been adopted  in (cf.~\cite{hu2018approximating, hu2019approximating}).  
 A recent work  by Glowinski {\it et al.} in \cite{glowinski2021bilinear} has conducted a numerical study on optimization algorithms for solving  problem \eqref{eq:p}.

Motivated by the need of reducing  the effects of rotation on the flow and the shear stress at the boundary in the cooling process, in this work  we are interested in minimizing  the magnitude of the strain tensor  (cf.~\cite{foias2001navier, liu2008mixing}), which is equivalent to minimize   $\|\nabla \bv\|_{L^2}$. 
In this case, we set
  \begin{align}
  U_{\text{ad}}= L^2(0, t_f;   V)
  \label{U}
\end{align} 
equipped with the norm
$\|\bv\|_{U_{\text{ad}}}=\|\bv\|_{L^2(0, t_f; H^1(\Omega))}.$
The  regularity of  $U_{\text{ad}}$ defined by \eqref{U} will allow  us  to carry out the G$\hat{a}$teaux  differentiability of the state equations.   Then  the optimality conditions can be established by directly employing   a variational inequality  or the Euler-Lagrange method.
However,  to numerically implement  the resulting optimality system for problem \eqref{eq:p}, one has to solve the state equations  forward in time, coupled with  the adjoint equations  backward in time 
together with a nonlinear optimality condition.  Straightforward use of this result  can result in extremely   high computational costs.  
Instantaneous control design is  a powerful tool for dealing  with the computational   limitations of open-loop control,  while providing a feedback law for flow control problems at a sustainable control cost (cf.~\cite{choi1993feedback, hinze2002analysis, hinze2000optimal, chang1999active, choi1999instantaneous, unger2001fast}).      The idea behind it is that it successively determines approximations of the objective function while marching forward in time.  The uncontrolled dynamical system is first discretized in  time. Then, at selected time slices an instantaneous version of the cost functional is approximately minimized subject to a stationary system, whose structure depends on the chosen discretization method. The control obtained is used to steer the system to the next time slice, where the procedure is repeated (cf.~\cite{hinze2002analysis}). 
This method  is closely tied to receding horizon control (RHC) or model predictive control (MPC) with finite time horizon  (cf.~\cite{garcia1989model, nevistic1997finite, rawlings1993stability, bewley2001dns}).   
Essentially,  instantaneous control is a discrete-in-time and suboptimal
 feedback control approach 
 and  can be interpreted as the stable time discretization of a closed-loop control law (cf.~\cite{choi1995suboptimal, lee1998suboptimal, hinze2002analysis, hinze2000optimal,min1999suboptimal, hinze1998control}).
On the other hand, given the optimality system, it is natural  to ask whether it is possible to obtain the equivalent   feedback laws by  first solving it  restricted to each time slice and then marching forward in time.  Following  the convention, without any ambiguity, we will  call the former  ``discretize-then-optimize (DTO)" approach and the latter ``optimize-then-discretize (OTD)"  approach in this work.

The remainder of this paper is organized as follows.  In section \ref{opt_design},    the first  order optimality  conditions are established  for  solving an optimal solution  using a variational inequality (cf.~\cite{ lions1971optimal}). 
Then the second order necessary conditions are  derived to charactering the solution
when the control weight $\gamma$ is sufficiently large.
In section \ref{feedback},  the feedback control  is constructed using both  DTO and OTD approaches, which turn out to be   the same feedback law under  appropriate discretization schemes.  The well-posedness and asymptotic behavior  of the closed-loop system will be also  addressed.  
  Numerical implementation of our control designs are presented in section \ref{Num}, 
where several numerical experiments  are conducted  to compare  the effectiveness of the optimal control and the feedback control  for convection-cooling. 

In the sequel, the symbol $C$ denotes a generic positive constant, which is allowed to depend on the domain as well as on indicated {parameters without ambiguous.}

\section{Existence and Optimality Conditions} \label{opt_design}
In this section, we discuss   the existence of an optimal solution to problem \eqref{eq:p} and derive the first and second order optimality conditions for characterizing  the optimal control by utilizing  a variational inequality (cf.~\cite{lions1971optimal}). \begin{thm}\label{thm1} 
For $T_0\in L^\infty(\Omega)$, there exists at least  one optimal  solution $\bv\in U_{\text{ad}}$ to  problem \eqref{eq:p}.
 \end{thm}
 The proof of the existence  for $\bv\in U_{\text{ad}}$ follows the similar approaches as in   \cite[Theorem 1]{barbu2016optimal} for $\bv\in L^2(0, t_f; H)$. The details are omitted here. 
 To establish the optimality conditions, however, it is critical to understand the regularity properties of the solution to the state equations for  $\bv\in U_{\text{ad}}$. 
 
 The following  results  will be often used  in this work.  The detailed proof of next lemma can  be found in (cf.~\cite{Te1997}).
 \begin{lemma}\label{lem0}
Let $\mathbf{v}\in L^2(0; t_f; H^1(\Omega), d=2,3$, $\phi\in L^2(0; t_f; H^1(\Omega))$, and  $\psi\in H^1(\Omega)$.
Then  we have 
\begin{align}
&\left|\int_{\Omega} (\mathbf{v} \cdot \nabla \phi) \psi\,dx\right|\leq \|\mathbf{v}\|_{L^4}\|\nabla \phi\|_{L^2}\|\psi\|_{L^4}
\leq C\|\nabla \mathbf{v}\|_{L^2}\|\nabla \phi\|_{L^2}\|\nabla \psi\|_{L^2},  \label{1EST_tri}
\end{align}
and hence,  
\begin{align}
\bv\cdot \nabla \phi\in L^1(0, t_f; (H^1(\Omega))').\label{EST_force}
\end{align}
Moreover, if $\nabla \cdot \mathbf{v}=0$ and $\mathbf{v}|_{\Gamma}=0$, then 
\begin{align}
\int_{\Omega} (\mathbf{v} \cdot \nabla \phi) \psi\,dx=-\int_{\Omega} \phi  \mathbf{v} \cdot \nabla\psi\,dx.
\label{2EST_tri}
\end{align} 
\end{lemma}

In addition,  since the velocity field is incompressible with no-slip boundary condition,  it is easy to check  that given  zero Neumann boundary condition,  the average of the temperature satisfies 
\begin{align}
\langle T\rangle=\langle T_0\rangle, \quad \forall t\in[0, t_f]. \label{ave_T}
\end{align} 
 In fact, taking the integral  of \eqref{sta_T} over $\Omega$  and applying Stokes formula  (cf.~\cite{Te1997}) together with \eqref{sta_v}--\eqref{BC} yields
 \begin{align*}
	\frac{d}{d t}\left(\int_{\Omega}T\,dx \right)&=\kappa\int_{\Omega} \Delta T\,dx-\int_{\Omega}\bv \cdot \nabla T\,dx
	=\kappa\int_{\Gamma} \frac{\partial T}{\partial n}\,dx-\int_{\Gamma}( \bv \cdot n)  T\,dx+\int_{\Omega}(\nabla \cdot \bv )T\,dx =0,
	\end{align*}
and therefore \eqref{ave_T} follows. 
	
\begin{lemma} \label{lem1}
  Let  $T_0\in L^\infty(\Omega)\cap H^1(\Omega)$. 
  For $\bv\in U_{\text{ad}}$,    there exists a unique solution  
to  the state equations  \eqref{sta_T}--\eqref{BC}, which satisfies  
\begin{align}
T\in (L^\infty(0, t_f; L^\infty(\Omega) \cap H^1(\Omega)) \cap L^2(0, t_f; H^2(\Omega)). \label{EST_TH3}
\end{align}


   \end{lemma}
\begin{proof}
For $T_0\in L^\infty(\Omega)$ and $\bv\in L^2(0, t_f; H)$, the existence of a unique weak solution  $T$ to \eqref{sta_T}--\eqref{IC} has been shown in \cite[Theorem 1]{barbu2016optimal}. Moreover, 
\begin{align}
T\in C([0, t_f]; L^2(\Omega))\cap L^2(0, t_f; H^1(\Omega))\cap L^{\infty}(0, t_f; L^\infty(\Omega)).\label{EST_TL2}
\end{align}  
To see \eqref{EST_TL2},  taking  the inner product of \eqref{sta_T} with $T$ and integrating  by parts using \eqref{BC}, we have
\begin{align}
&\frac{1}{2}\frac{d\|T\|^2_{L^2}}{dt}+\kappa \|\nabla T\|^2_{L^2} =-\int_{\Omega} (\bv\cdot \nabla T)T\, dx
=-\frac{1}{2}\int_{\Omega}\bv\cdot  \nabla (T^2)\,dx
 = -\frac{1}{2}\left( \int_{\Gamma}(\bv\cdot  n) \, T^2\,dx -\int_{\Omega}(\nabla \cdot \bv) \,T^2\,dx\right)=0,
\label{1EST_TL2}
\end{align}
which  gives
   \begin{align}
\| T\|^2_{L^2}+2\kappa\int^{t}_0 \|\nabla T\|^2_{L^2}\,dt = \|T_0\|^2_{L^2}, \quad t\in [0, t_f].
\label{EST_TH1}
\end{align}
Furthermore, since $\frac{\partial T}{\partial t}\in L^2(0, t_f; (H^1(\Omega))')$, by  Aubin-Lions Lemma we have $T\in C([0, t_f]; L^2(\Omega))$.

Analogously, taking the inner product of  \eqref{sta_T} with $T^{N-1}$ with $N\ge 2$ and then letting $N\to \infty$ we get
   \begin{align}
\sup_{t\in[0, t_f]}\|T\|_{L^\infty}\leq  \|T_0\|_{L^\infty}. \label{EST_Tinfty}
\end{align}
This estimate  can be achieved  by using the  Stampacchia  theory.  
The reader is referred to \cite{barbu2016optimal, stampacchia1965probleme} for details.
To see \eqref{EST_TH3}, taking the inner product of \eqref{sta_T} with $-\Delta T$  and using Green's formula follow
\begin{align}
&\frac{1}{2}\frac{d\|\nabla T\|^2_{L^2}}{dt}+\kappa \|\Delta T\|^2_{L^2} =-\int_{\Omega} \nabla (\bv\cdot \nabla T)\cdot  \nabla T\, dx\label{0TH1}\\
&=-\int_{\Omega} \partial_j v_i\partial_i T\partial_j T\,dx- \frac{1}{2} \int_{\Omega}  v_i\partial_i (\partial _jT\partial _jT)\, dx\label{1TH1}\\
&=-\int_{\Omega} \partial_j v_i\partial_i T\partial_j T\,dx- \frac{1}{2}\left( \int_{\Gamma}  v_in_i (\partial _jT\partial _jT)\, dx
-\int_{\Omega} \partial_i  v_i(\partial _jT\partial _jT)\, dx\right)\nonumber\\
&=-\int_{\Omega} \partial_j v_i\partial_i T\partial_j T\,dx
\leq \|\nabla \bv\|_{L^2}\|\nabla T\|^2_{L^4}
\leq  C \|\nabla \bv\|_{L^2}\|\nabla   T\|_{L^2}\|\Delta  T\|_{L^2}\nonumber\\
&\leq  C \|\nabla \bv\|^2_{L^2}\|\nabla   T\|^2_{L^2}+\frac{\kappa}{2}\|\Delta  T\|^2_{L^2},\label{2TH1}
\end{align}
where from  \eqref{0TH1}  to \eqref{1TH1} we used  Einstein's summation convection, i.e.,  $ \nabla (\bv\cdot \nabla T)\cdot  \nabla T=\partial_j(v_i\partial_i T)\partial_j T$. 
From \eqref{2TH1} we get 
\begin{align}
&\frac{d\|\nabla T\|^2_{L^2}}{dt}+\kappa \|\Delta T\|^2_{L^2} 
\leq  C \|\nabla \bv\|^2_{L^2}\|\nabla   T\|^2_{L^2}, \label{00EST_TH1}
\end{align}
and hence, using Gr\"{o}nwall's inequality gives
\begin{align}
&\sup_{t\in[0, t_f]}\|\nabla   T\|_{L^2} \leq  e^{C \int^{t_f}_0 \|\nabla \bv\|^2_{L^2}\,dt}\|\nabla   T_0\|_{L^2}<\infty.\label{0EST_TH1}
\end{align}
Moreover,  from \eqref{00EST_TH1} we have
\begin{align*}
\kappa \int^{t_f}_0\|\Delta T\|^2_{L^2} \,dt
\leq  C \int^{t_f}_0\|\nabla \bv\|^2_{L^2}\|\nabla   T\|^2_{L^2}\,dt
\leq C\| \bv\|^2_{U_{\text{ad}}} \sup_{t\in[0, t_f]}\|\nabla   T\|_{L^2}<\infty,
\end{align*}
which completes the proof.
\end{proof}


\subsection{ Optimality Conditions}\label{opt}

Let $A=-\mathbb{P}\Delta$ be the Stokes operator   with
 $
 D(A)=V\cap H^2(\Omega),
$
where $\mathbb{P}\colon L^2(\Omega)\to H$ is the Leray projector (cf.~\cite[p.\,31]{constantin1988navier}).  Note that   $A$ is a  strictly positive and self-adjoint operator.  Moreover, 
define $D\colon L^2(\Omega)\to L^2(\Omega)$  such that $DT = T - \langle T\rangle$. Then the cost functional is equivalent   to  \begin{align}
 J(\bv)&=\frac{\alpha}{2}\|DT(t_f)\|^2_{L^2}+ \frac{\beta}{2}\int^{t_f}_0(D^*DT, T)\,dt +\frac{\gamma}{2}\int^{t_f}_0 (A\bv,\bv )\,dt.\label{cost}
  \end{align}
 As shown in \cite{he2021optimal},  it is easy to versify that $D=D^*$ and $D^2=D$, thus $\|D\|\leq 1$.

Now we derive the first order  necessary  optimality conditions  for problem \eqref{eq:p} by using a variational inequality  (cf.~\cite{lions1971optimal}), that is,  
if $\bv$ is an optimal solution of problem \eqref{eq:p}, then there holds
\begin{align} 
J'(\bv)\cdot (\mathbf{w}-\bv)\geq 0, \quad  \mathbf{w}\in U_{\text{ad}}. \label{var_ineq}
\end{align}
To establish the  G$\hat{a}$teaux differentiability of  $J(\bv)$, we first check the G$\hat{a}$teaux differentiability of $T$ with respect to $\bv$. 
Let  $z$ be the G$\hat{a}$teaux derivative of $T$ with respect to $\bv$ in the direction of $h\in U_{\text{ad}}$, i.e., $z = T'(\bv )\cdot h$. Then $z$ satisfies 
 \begin{equation}\label{z-equation}
	\begin{split}
		\frac{\partial z}{\partial t}&=\kappa \Delta z-\bv \cdot \nabla z-h \cdot \nabla T,\quad \frac{\partial z}{\partial n}\Big|_{\Gamma}=0, 
	\end{split}
\end{equation}
with $z(x, 0)=0$.
To show existence of  \eqref{z-equation}, we first establish an  {\it a prior} estimate of $z$. 
Taking the inner product of \eqref{z-equation} with $z$ and applying \eqref{2EST_tri}, we get
 \begin{align*}
\frac{1}{2}\frac{d\|z\|^2_{L^2}}{dt}+ \kappa \|\nabla z\|^2_{L^2}
&=\int_{\Omega} T (h\cdot \nabla z)\,d x
\leq \| T\|_{L^\infty}  \|h\|_{L^2} \|\nabla z\|_{L^2}\nonumber\\
&\quad\leq \frac{1}{2\kappa}\| T\|^2_{L^\infty}  \|h\|^2_{L^2} +\frac{\kappa}{2}\|\nabla z\|^2_{L^2},
\end{align*}
which follows
 \begin{align*}
&\frac{d\|z\|^2_{L^2}}{dt}+ \kappa \|\nabla z\|^2_{L^2}
\leq \frac{1}{\kappa}\| T\|^2_{L^\infty}  \|h\|^2_{L^2}.
\end{align*} 
With the help of Lemma \ref{lem0} and \eqref{EST_TH1} we have 
\begin{align}
& \|z\|^2_{L^2} + \kappa \int^{t}_0 \|\nabla z\|^2_{L^2} \,ds
\leq  \frac{1}{\kappa}  \int^{t}_0\| T\|^2_{L^\infty}  \|h\|^2_{L^2}\,ds
\leq   \frac{C}{\kappa} \|T_0\|^2_{L^\infty} \| h\|^2_{U_{\text{ad}}}, \quad t\in [0, t_f]. \label{EST_zL2}
\end{align}
Based on Lemma \ref{lem0}, \eqref{EST_TH1} and \eqref{EST_zL2}, it is clear that   $\bv \cdot \nabla z$ and  $h \cdot \nabla T\in L^1(0, t_f; H^{-1}(\Omega))$, thus $\frac{\partial z}{\partial t}\in L^1(0, t_f; H^{-1}(\Omega))$. According to \cite[Theorem 3.1]{Te1997},
there exists a unique  solution to \eqref{z-equation}  and $z\in L^{\infty}(0, t_f; L^2(\Omega))\cap L^2(0, t_f; H^1(\Omega))$.
Therefore, $T(\bv)$ is  G$\hat{a}$teaux differentiable for $\bv\in U_{\text{ad}}$, so is $J(\bv)$.

The following theorem establishes the 
first order optimality conditions for the solving the optimal control.

\begin{thm}\label{thm1}
If $\bv$ is the optimal solution to problem \eqref{eq:p} and $T$ is the corresponding solution to the state equations \eqref{sta_T}--\eqref{IC}. Then there exists an adjoint state  $q$ such that  the optimal triplet $(\bv, T, q)$ satisfies 
\begin{align}
&	\frac{\partial T}{\partial t}=\kappa \Delta T-\bv \cdot \nabla T, \quad \frac{\partial T}{\partial n}\Big|_{\Gamma} = 0, \quad T(0)=T_0,  \label{opt_T}\\
& -\frac{\partial q}{\partial t}=\kappa \Delta q+{\bv}  \cdot \nabla   q  + \beta D^*DT, \quad	\frac{\partial q}{\partial n}\Big|_{\Gamma} = 0, \quad q(t_f)=\alpha D^*DT(t_f),\label{opt_q}\\
&-\gamma\Delta {\bv}+\nabla p = q \nabla  T,\quad \nabla\cdot \bv = 0, \quad \bv|_{\Gamma} = 0,\label{opt_cond}
\end{align}
where pressure  $p\in L^2(\Omega)$   satisfies $\int_{\Omega}p\,dx=0$.
\end{thm}
\begin{proof}
The first order necessary optimality system for $\bv\in L^2(0, t_f; H)$  has been derived  in  \cite[Theorem]{barbu2016optimal}  using an approximate control approach.
However,   since  $J$ is  G$\hat{a}$teaux differentiable for $\bv\in U_{\text{ad}}$ in our current work as shown in Theorem \ref{thm1},  we  are able to  directly  apply the variational inequality \eqref{var_ineq} to establish this result. 
 
First multiply \eqref{z-equation} by $q$ and integrate  over $\Omega\times (0, t_f)$. Then applying integration by parts and Green's formula together with  \eqref{2EST_tri},
we have 
\begin{align*}
(z(t_f), q(t_f))-\int^{t_f}_0(z, \frac{\partial q}{\partial t})\,dt&=\int^{t_f}_0(z, \kappa \Delta q)\,dt
+\int^{t_f}_0(z, {\bv}  \cdot \nabla   q)\,dt.
	\end{align*}
	On the other hand, 
	\begin{align}
	J'(\bv) \cdot h =&\alpha(D^*DT(t_f), z(t_f))+\beta\int^{t_f}_0(D^*DT, z)\,dt + \gamma\int^{t_f}_0 (A \bv, h)\,dt. \label{J_1st}
\end{align}
Now let  the adjoint state $q$	satisfy   \eqref{opt_q}.
The  G\^{a}teaux  derivative  of $J$ becomes 
\begin{align}
		J'(\bv) \cdot h 
		=&( q(t_f), z(t_f))- \int^{t_f}_0(\frac{\partial q}{\partial t}+\kappa \Delta q  + \bv \cdot \nabla q, z)\,dt + \gamma\int^{t_f}_0 (A\bv, h)\,dt\nonumber\\
		=& \int^{t_f}_0 (q,   \frac{\partial z}{\partial t} - \kappa \Delta  z+ \bv\cdot  \nabla z) \,dt+  \gamma\int^{t_f}_0(A\bv, h)\,dt\nonumber\\
		=& \int^{t_f}_0 (q,  -h \cdot \nabla T) \,dt+  \gamma \int^{t_f}_0(A\bv, h)\,dt\nonumber\\
		=& \int^{t_f}_0 (-q\nabla T, h)\,dt +  \gamma\int^{t_f}_0 (A\bv, h)\,dt. \label{EST_J}
\end{align}
Therefore, if $\bv^{\text{opt}}$ is the optimal solution, then $J'(\bv^{\text{opt}})\cdot h\geq 0$ for any $h\in U_{\text{ad}}$. This implies 
\begin{align}
 \gamma A\bv^{\text{opt}}-\mathbb{P}(q\nabla T)=0\quad \text{or}\quad
-  \gamma\Delta  \bv^{\text{opt}}+\nabla p-q\nabla T=0
  \label{opt_v}
\end{align}
for some $p\in L^2(\Omega)$ with  $\int_{\Omega}p\,dx=0$. 

Moreover, applying the similar approaches as in Lemma \ref{lem1} and \eqref{EST_Tinfty} and noting that $\|D\|\leq 1$, we have
	\begin{align}
&	\| q\|^2_{L^2}+2\kappa\int^{t}_0 \|\nabla q\|^2_{L^2} 
\leq \beta \int^{t_f}_0\|T\|^2_{L^2}\, dt+ \alpha\|T(t_f)\|^2_{L^2}
\leq C(T_0, t_f), \quad t\in [0, t_f],  \label{EST_qL2}\\
&\text{and} \quad \sup_{t\in [0, t_f]} \|q\|_{L^\infty}\leq \beta \int^{t_f}_0\|T\|_{L^\infty}\,dt+\alpha \|T(t_f)\|_{L^\infty}
\leq C(T_0, t_f),  \label{EST_qinfty}
\end{align}	
for some constant $C(T_0, t_f)$ depending on $T_0$ and $t_f$. 
 This completes the proof.
\end{proof}

Note that the uniqueness of the solution to the optimality system \eqref{opt_T}--\eqref{opt_cond} can be  obtained  under certain conditions on $T_0, t_f$ and $\gamma$. The  proof follows the same as in  \cite[Theorem 6]{barbu2016optimal}.  Moreover,  we have   the following regularity  results  for any optimal triplet 
$(\bv, T,q)$ satisfying \eqref{opt_T}--\eqref{opt_cond}. The proof is presented in Appendix \ref{app}.

\begin{corollary}\label{cor1}
If  $T_0\in H^2(\Omega)$ and  $(\bv, T, q)$ satisfies  the first order necessary optimality system  \eqref{opt_T}--\eqref{opt_cond}, then
\begin{align}
\bv\in L^\infty(0, t_f; V\cap H^2(\Omega)), \quad T\in L^{\infty}(0, t_f; H^2(\Omega))\cap L^2(0, t_f; H^3(\Omega)).
\end{align}
\end{corollary}


With the help of these properties,  we can further  address   the  second order necessary optimality conditions for characterizing the optimal solutions.

\begin{thm}\label{SONC}
Let  $\bv$ be an  optimal solution to problem \eqref{eq:p} and  the  triplet $(\bv, T, q)$ satisfy   the first order necessary optimality system  \eqref{opt_T}--\eqref{opt_cond}.  If $\gamma>0$ is sufficiently large,  then  there exists some constant $c_0>0$ such that 
\begin{align}
 J''(\bv)\cdot (h, h) \geq c_0\|h\|^2_{U_{\text{ad}}},   \label{opt_2rd}
\end{align}
for $ h\in U_{\text{ad}}$.
\end{thm}
The  proof of Theorem \ref{SONC} is given  in   Appendix \ref{app}.  However,  the regularity of $U_{\text{ad}}$ is not sufficient for $J$ to have the twice G$\hat{a}$teaux differentiability in general.

\section{Feedback Control Law Based on Instantaneous Control Design }\label{feedback}

With the understanding  of the optimal control  design in our disposal, we are in the position to construct a feedback control law based on the method of instantaneous control and compare  the DTO approach with the OTD approach.  The former, as mentioned earlier,  is to first discretize the uncontrolled state equations in time and conduct  the optimization procedure over discrete time steps, and then progress  recursively in time (cf.~\cite{hinze2000optimal, hinze2002analysis}). In contrast, the latter  is to directly discretize the optimality system \eqref{opt_T}--\eqref{opt_cond} on one step time sub-interval, and then carry the information for  the next time sub-interval, where the state and the adjoint equations will be formulated forward and backward in time, respectively, but just for one step.  Finally,  we observe that under appropriate time discretization schemes, these two approaches  lead to the same nonlinear continuous feedback controller. Its effectiveness  will be compared with the optimal control numerically  in section \ref{Num}.

\subsection{Discretize-then-Optimize Approach}
Consider a uniform partition of $[0, t_f]$ and let $\tau=\frac{t_f}{n+1}$ for $n\in \mathbb{N}$ and $t_i=i\tau, i=0,1, \dots,n$.
Using the  semi-implict Euler's method for discretizing the state equations \eqref{sta_T} in time gives, for $i=0, 1,\dots, n$, 
  \begin{align} 
      & \frac{T^{i+1}-T^i}{\tau}=\kappa \Delta T^{i+1}-\bv^{i+1}\cdot \nabla T^i, 
      \quad \text{that is} \quad (I-\kappa \tau\Delta)T^{i+1}=T^i- \tau\bv^{i+1}\cdot \nabla T^i,  \label{T_dis}
   \end{align}
   where $T^0=T_0$.
Let  $\alpha=0$, $\beta=1$, $U^{i}_{\text{ad}}=V$, and $\langle T^{i+1}\rangle=\frac{1}{\tau }\int^{t_{i+1}}_{t_i}\langle T\rangle\,ds$. Given $T^i$ at $t_i$, we solve for the control $\bv^{i+1}$ at $t_{i+1}$ by minimizing the following instantaneous version of the cost functional $J$ in ($P$):
   \begin{align*}
J^{i+1}(\bv^{i+1})&=\frac{1}{2} \int_{\Omega} |T^{i+1}-\langle T^{i+1}\rangle|^{2}\,dx
+\frac{\gamma}{2} \int_{\Omega} |A^{1/2} \bv^{i+1}|^{2}\, dx\nonumber\\
&=\frac{1}{2} (D^*DT^{i+1}, T^{i+1})+\frac{\gamma}{2}(A\bv^{i+1}, \bv^{i+1})\hspace{1in}\tag{$P^{i+1}$} \label{eq:pi}  
\end{align*}
subject to \eqref{T_dis}.
Again using a similar variational inequality as shown in proof of Theorem \ref{thm1}, we have
   \begin{align*}
(J^{i+1})'(\bv^{i+1})\cdot h^{i+1}&=(D^*DT^{i+1}, z^{i+1})+\gamma(A\bv^{i+1}, h^{i+1})
\end{align*}
for $h^{i+1}\in U^{i+1}_{\text{ad}}$, where $z^{i+1}=(T^{i+1})'(\bv^{i+1})\cdot h^{i+1}$ satisfies 
  \begin{align} 
       (I-\kappa \tau\Delta)z^{i+1}=- \tau h^{i+1}\cdot \nabla T^i,  \quad i=0, 1,\dots, n. \label{z_dis}
   \end{align} 
Define the adjoint state $q^{i+1}$ such that 
   \begin{align}
       & (I- \kappa\tau  \Delta)q^{i+1}=D^*DT^{i+1}. \label{q_dis}
             \end{align}
      Then  with the help of \eqref{z_dis}--\eqref{q_dis}, we get
        \begin{align*}
(J^{i+1})'(\bv^{i+1})\cdot h^{i+1}&=( (I- \kappa\tau  \Delta)q^{i+1}, z^{i+1})+\gamma(A\bv^{i+1}, h^{i+1})\\
&= -(q^{i+1},\tau h^{i+1}\cdot \nabla T^i)+\gamma(A\bv^{i+1}, h^{i+1}) \\
&= -(\tau q^{i+1} \nabla T^i,h^{i+1})+\gamma(A\bv^{i+1}, h^{i+1}),
\end{align*}
which implies that if $\bv^{i+1}$ is an optimal solution to problem $\eqref{eq:pi}$, then it satisfies a Stokes equation
        \begin{align}
    \gamma A\bv^{i+1}-\tau \mathbb{P}(q^{i+1} \nabla T^{i})=0\quad 
\text{or}\quad
-  \gamma\Delta  \bv^{i+1}+\nabla p^{i+1}-\tau q^{i+1}\nabla T^{i}=0, \quad i=0, 1,\dots,n,
 \label{opt_vi}
\end{align}
for some $p^{i+1}\in L^2(\Omega)$ with  $\int_{\Omega}p^{i+1}\,dx=0$.

Let 
$E_{\tau}= I-\kappa \tau\Delta$ with domain 
$D(E_{\tau})=\{ T\in H^2(\Omega)\colon   \frac{\partial T}{\partial n}|_{\Gamma}=0\}$.  Then $E_{\tau}$ is a strictly positive elliptic operator for $\kappa \tau>0$. 
In summary, the optimality system for  problem $\eqref{eq:pi}$ is governed by, for $i=0,1, \dots n$, 
\begin{eqnarray}\label{opt_syst_dis}
\begin{cases}
E_{\tau} T^{i+1}=T^i- \tau\bv^{i+1}\cdot \nabla T^i, \quad \frac{\partial T^{i+1}}{\partial n}|_{\Gamma}=0,\\
	E_{\tau} q^{i+1}=D^*DT^{i+1},  \quad \frac{\partial q^{i+1}}{\partial n}|_{\Gamma}=0,\\
      -\gamma\Delta \bv^{i+1}+\nabla p^{i+1} = \tau q^{i+1} \nabla  T^{i},\quad \nabla\cdot \bv^{i+1} = 0, \quad \bv^{i+1}|_{\Gamma} = 0.
\end{cases}
\end{eqnarray}
 The optimality system \eqref{opt_syst_dis}  admits a unique solution due to the quadratic cost functional and the uniqueness of solution to the  discretized state equation \eqref{T_dis}.

 To construct a feasible  feedback control law based on the nonlinear optimality system \eqref{opt_syst_dis}, we suggest  first solving  $q^{i+1}=E^{-1}_{\tau}D^*D T^{i+1}$ from the second equation, and then obtain an implicit  approximation to  $\bv^{i+1}$ from the third equation
\begin{align}
		-\gamma\Delta \bv^{i+1}+\nabla p^{i+1} = \tau (E^{-1}_{\tau}D^*D  T^{i+1}) \nabla  T^{i},\quad \nabla\cdot \bv^{i+1} = 0, \quad \bv^{i+1}|_{\Gamma} = 0,\label{v_implict}
\end{align}
or equivalently 
$\bv^{i+1}=\frac{\tau}{\gamma} A^{-1}\mathbb{P} (E^{-1}_{\tau}D^*D  T^{i+1} \nabla  T^{i})$.
Upon plugging this implicit instantaneous control $\bv^{i+1}$ into the first equation, we get an implicit time marching scheme from $T^{i}$ to $T^{i+1}$:
\begin{align*}
(I-\kappa \tau\Delta)T^{i+1}=E_{\tau} T^{i+1}&=
      T^i-  \tau\frac{\tau}{\gamma} [A^{-1}\mathbb{P} ((E^{-1}_{\tau}D^*D T^{i+1})\nabla T^{i})]\cdot \nabla T^i  \quad i=0,1,\dots, n.
            \end{align*}
The above nonlinear scheme is not suitable for computation, but it turns out to be  a semi-implicit time  discretization (with the time step size $\tau$)  of 
 a closed-loop dynamical system (retain $\tau$ as a parameter )
            \begin{align}
                  \frac{\partial T}{\partial  t} 
      &=\kappa\Delta T-   \underbrace{\frac{\tau}{\gamma} [A^{-1}\mathbb{P} ((E^{-1}_{\tau}D^*DT)\nabla T)}_{\bv}]\cdot \nabla T, \quad T(0)=T_0,
      \label{closed_loop}
            \end{align}   
where the continuous control $\bv$ is given by the nonlinear feedback law :
\begin{align}
&\bv= \frac{\tau}{\gamma} A^{-1}\mathbb{P} ((E^{-1}_{\tau} D^*DT)\nabla T)
\quad \text{or}\quad
-  \gamma \Delta  \bv+\nabla p=\tau( (E^{-1}_{\tau} D^*DT)\nabla T).  \label{feedback_v}
\end{align}
Although no theoretical guarantee in optimality,  we examine the  performance of the feedback law in minimizing the objective functional $J$ numerically, which can be computed  much more efficiently than the optimal control. 
  \begin{remark}
  Note that if solving velocity explicitly in \eqref{feedback_v} using  $T^{i+1}=E^{-1}_{\tau}T^i$ and $\bv^i_0=0$ for each iteration, we would have
  \begin{align}
&\bv= \frac{\tau}{\gamma} A^{-1}\mathbb{P} ((E^{-1}_{\tau}(D^*DE^{-1}_{\tau} T))\nabla T)
\quad \text{or}\quad
-  \gamma \Delta  \bv+\nabla p=\tau( (E^{-1}_{\tau} D^*DE^{-1}_{\tau} T)\nabla T), \label{2feedback_v}
\end{align}
which involves a more regularized   $T$ compared to \eqref{feedback_v}. 
Also,   the gradient decent method is not used  for solving $\bv^{i+1}$ as in \cite{hinze2000optimal,hinze2002analysis}, yet    the optimality condition \eqref{opt_vi} is directly called. This way will keep the control weight $\gamma$ in the closed-loop system. 
By properly choosing this parameter and step size $\tau$, one can establish the well-posedness and stability of the closed-loop system (see Theorem \ref{stability}).
  	Moreover, once  the continuous  closed-loop dynamical system is derived, $\tau$ only plays a role as a parameter associated with the feedback control law.  It does not indicate  the time step size in the numerical simulation of the nonlinear closed-loop system.  
\end{remark}  
  
\subsection{Optimize-then-Discretize Approach}

Alternatively, motivated by the idea of instantaneous control, we consider a direct application  of the optimality system \eqref{opt_T}--\eqref{opt_cond}  derived  in Theorem \ref{thm1} to formulate the feedback law.   To this end, letting  $\tau=t_f/(n+1)$ and $t_i=i\tau, i=0,1,\cdots,n+1$, we divide the global time interval $[0,t_f]$ into uniformly spaced sub-intervals $I_i=[t_i,t_{i+1}]$, and then solve  the continuous optimal control problem (P) restricted to each interval $I_i$ sequentially, where for $i\ge 1$ the initial condition of $T$ on $I_i$ is given by the solution from the previous sub-interval $I_{i-1}$. 
Let $T|_{I_i},q|_{I_i},\bv|_{I_i}$ denotes the desired continuous  state, adjoint state, and optimal control on each sub-interval $I_i$, respectively.
According to  Theorem \ref{thm1}, the localized optimality system defined on $I_i$ reads (only consider the case $\alpha=0,\beta=1$)
\begin{align}
	&	\frac{\partial T|_{I_i}}{\partial t}=\kappa \Delta T|_{I_i}-\bv|_{I_i} \cdot \nabla T|_{I_i}, \quad \frac{\partial T|_{I_i}}{\partial n}\Big|_{\Gamma} = 0, \quad T|_{I_i}(\cdot,t_i)=T|_{I_{i-1}}(\cdot,t_i),   \label{opt_T_i}\\
	& -\frac{\partial q|_{I_i}}{\partial t}=\kappa \Delta q|_{I_i}+{\bv|_{I_i}}  \cdot \nabla   q|_{I_i}  +  D^*DT|_{I_i}, \quad	\frac{\partial q|_{I_i}}{\partial n}\Big|_{\Gamma} = 0, \quad q|_{I_i}(t_{i+1})=0,\label{opt_q_i}\\
	&-\gamma\Delta {\bv|_{I_i}}+\nabla p|_{I_i} = q|_{I_i} \nabla  T|_{I_i},\quad \nabla\cdot \bv|_{I_i} = 0, \quad \bv|_{I_i}|_{\Gamma} = 0,\label{opt_cond_i}
\end{align}
where all involved variables are continuously defined  in $I_i$ only.
For simplicity, we will drop the restriction notation $|I_i$ 
in the following time discretization scheme on $I_i$.
Let $T^{i},T^{i+1}, q^i,q^{i+1}, \bv^i,\bv^{i+1}$  denote the finite difference approximation to $T,q,\bv$ at the two end points $t_i,t_{i+1}$ of the sub-interval $I_i$, respectively.
Applying a semi-implicit Euler time scheme with the same step size $\tau$ to the  localized optimality conditions  on $I_i$ we obtain 
a semi-discretized optimality system (dropped the cumbersome notation $|I_i$)
\begin{align}
	& 
     E_{\tau}T^{i+1}=T^i- \tau\bv^{i+1}\cdot \nabla T^i,\quad
      \quad  \frac{\partial T^{i+1}}{\partial n}\Big|_{\Gamma}=0, \quad T^i=T|_{I_{i-1}}(\cdot,t_i) \label{OTD_T}\\ 
	& E_{\tau} q^{i}=q^{i+1}+\tau ( \bv^{i+1}  \cdot \nabla   q^{i+1}  + D^*DT^{i+1}),\quad
	  \frac{\partial q^{i}}{\partial n}\Big|_{\Gamma}=0, \quad q^{i+1}=0, \label{OTD_q}\\\
&-	\gamma\Delta \bv^{i+1}+\nabla p^{i+1} = q^{i} \nabla  T^{i},\quad \nabla\cdot \bv^{i+1} = 0,
 \quad \bv^{i+1}|_{\Gamma} = 0.
\label{OTD_cond}
\end{align}
 Here  the adjoint state $q$  is defined only locally on each time sub-interval $I_i$,  which is different from the global adjoint state on $[0,t_f]$. {The semi-implicit scheme is also applied for the nonlinear term $q\nabla T$ on the right hand side of the optimality condition \eqref{opt_cond_i}. Specifically,  $q$ on the right-hand-side of \eqref{OTD_cond} is chosen to be on $t_i$, which will be  solved backward in $i$. }
  In fact, 
 from \eqref{OTD_q}, using $q^{i+1}=0$  we obtain
$
 q^i=\tau E_\tau^{-1} (D^*DT^{i+1}). \label{2OTD_q}
$
Therefore, the optimality condition becomes
 \begin{align}
- \gamma\Delta \bv^{i+1}+\nabla p^{i+1} & =\tau (E^{-1}_{\tau} D^*DT^{i+1})\nabla  T^{i} , \label{OTD_feedback}
 \end{align}
 which results in  the same nonlinear feedback law as in \eqref{feedback_v}  and so is the closed-loop system \eqref{closed_loop}.
 Such an equivalence is due to  the particular semi-discretization schemes we used in derivation, however,  the outcome may be quite different with other semi-discretization schemes.
\begin{remark}
We notice that the time discretization scheme of the state equations  determines  the resulting  feedback law, how to effectively handle the  discretization of the advective term   is the key in the instantaneous design for this  type of  bilinear control problems. If  a fully implicit time discretization was  applied, it  would generate  a more complicated nonlinear feedback law that causes an additional layer of difficulty  in analyzing   the closed-loop system. 
It is in general also difficult to estimate  the performance of such feedback laws.
\end{remark}

 \subsection{Well-posedness and Asymptotic Behavior   of the Closed-Loop System }
 
 First recall that   the  incompressible velocity field  neither engenders   energy to the system nor  consumes any via pure advection as time evolves. 
The variance    $\|DT\|_{L^2}$  decays exponentially    due to dissipation alone (see Remark \ref{prop_aveT} in Appendix \ref{app}).  
 However,    the feedback law  does  help enhance  cooling or  homogenization of the temperature distribution shown in our numerical experiments  as well as   quantified  by the ``mix-norm" (see Remark \ref{mix}).
Without loss of generality, we assume $\langle T_0\rangle=0$ in the rest of our discussion, then by \eqref{ave_T} we have  $\langle T\rangle=0$ for any $t\in[0, t_f]$. Thus  $D^*DT=T$.  Also,
since   
$$ \mathbb{P} ((E^{-1}_{\tau} T)\nabla T) =\mathbb{P}(  \nabla( (E^{-1}_{\tau}T)T))- T \nabla (E^{-1}_{\tau}T  ))
=- \mathbb{P} ( T \nabla (E^{-1}_{\tau}T  ) ), $$
 the closed-loop system \eqref{closed_loop} becomes 
            \begin{align}
                  \frac{\partial T}{\partial  t} 
      &=\kappa\Delta T+ \frac{\tau}{\gamma} A^{-1}\mathbb{P} (T \nabla (E^{-1}_{\tau} T  ) )\cdot \nabla T, \quad T(0)=T_0.
      \label{2closed_loop}
            \end{align}
Let $\eta=E^{-1}_{\tau} T$ for any $T\in L^2(\Omega)$. Then it is easy to see that $\eta$  satisfies 
\begin{align}
&E_{\tau}\eta=(I-\kappa\tau \Delta)\eta=T, \quad \frac{\partial \eta}{\partial n}\Big|_{\Gamma}=0, \label{eta}
\end{align}
and 
\begin{align}
\|\nabla \eta\|^2_{L^2}
\leq \frac{1}{2\kappa\tau}\|T\|^2_{L^2}.\label{2eta}
\end{align}
With the help of \eqref{2eta} and \eqref{EST_TH1}--\eqref{EST_Tinfty},  we have 
  \begin{align}
  \|A\bv\|^2_{L^2}
   &= \frac{\tau^2}{\gamma^2} \| \mathbb{P} (T \nabla (E^{-1}_{\tau}T) )   \|^2_{L^2}
  \leq C \frac{\tau^2}{\gamma^2} \|T\|^2_{L^{\infty}}\|\nabla (E^{-1}_{\tau}T) )\|^2_{L^2}
 \leq   \frac{C\tau}{\kappa \gamma^2}\|T_0\|^2_{L^{\infty}}\| T_0\|^2_{L^2}, 
  \end{align}
which implies
\begin{align}
 \sup_{t\in[0, t_f]} \|\bv\|^2_{H^2}
    \leq  \frac{C\tau}{ \kappa\gamma^2 }\|T_0\|^2_{L^{\infty}}\| T_0\|^2_{L^2}. \label{opt_vH2}
  \end{align}

Now we are ready to  address the well-posedness and asymptotic behavior  of the closed-loop system.

   \begin{thm}\label{stability}
For $T_0\in H^1(\Omega)\cap L^{\infty}(\Omega)$,  there exists a unique solution to \eqref{2closed_loop}.
Moreover,   if $\frac{\tau}{\gamma^2}$ is sufficiently small,  then there exists a constant $\delta_0>0$ such that 
   \begin{align}
 &  \|\nabla T\|^2_{L^2}\leq  e^{-\delta_0t}\|\nabla T_0\|^2_{L^2},  \label{0Exp_TH1}\\
\quad 
   & \int^{\infty}_0\|\Delta T\|^2_{ L^2}\,dt\leq  C (T_0,\kappa, \gamma, \tau).  \label{1Exp_TH1}
   \end{align}
  In addition,  if $T_0\in H^2(\Omega)$, then there exists  an constant $ \delta_1>0$ such that 
      \begin{align}
     &\|\Delta T\|^2_{L ^2}\leq  e^{-\delta_1t}\|\Delta T_0\|^2_{L^2},
     \label{0Exp_TH2}\\
     \quad 
     & \int^{\infty}_0 \|\nabla (\Delta T)\|^2_{L^2}\,dt\leq  C (T_0,\kappa, \gamma, \tau),
     \label{1Exp_TH2}
   \end{align}
   and 
    \begin{align}
	\|\frac{\partial T}{\partial t}\|_{L^2}
	& \leq  C (\kappa, \gamma, \tau) e^{-\max\{\delta_0, \delta_1 \}t} \|\Delta  T_0\|_{L^2}. \label{Exp_Tt}
\end{align}
\end{thm}

\begin{proof}
With the help of Lemma  \ref{lem1}, it suffices to show the uniqueness of the solution.  
We first assume that there are two solutions $T_1$ and $T_2$ satisfying 
 \eqref{2closed_loop} and let  $\bv_i$ be the velocity corresponding to $T_i, i=1,2$. Set $\theta=T_1-T_2$ and  $W=\bv_1-\bv_2$, then $\theta$ and $W$ satisfy  
    \begin{align}
    	\begin{split}
	&\frac{\partial \theta}{\partial t}=\kappa \Delta \theta-\bv_1 \cdot \nabla \theta-W\cdot \nabla T_2,
	\quad \frac{\partial \theta}{\partial n}\Big|_{\Gamma}=0, \label{closed_theta}\\
	&\theta(x,0)=0.
	\end{split}
\end{align}
Taking the inner product of \eqref{closed_theta} with $\theta$ follows
    \begin{align*}
	&\frac{1}{2}\frac{d\|\theta\|^2_{L^2}}{dt}+\kappa \|\nabla \theta\|_{L^2}=(-W\cdot \nabla T_2, \theta)=(T_2, W\cdot \nabla \theta)\\
	&\qquad\leq \| T_2\|_{L^\infty}\|W\|_{L^2}\|\nabla \theta\|_{L^2}
	\leq \| T_2\|^2_{L^\infty}\|W\|^2_{L^2}+\frac{\kappa}{2}\|\nabla \theta\|_{L^2}.
\end{align*}
Thus
     \begin{align}
	&\frac{d\|\theta\|^2_{L^2}}{dt}+\kappa \|\nabla \theta\|_{L^2}\leq C\| T_0\|^2_{L^\infty}\|W\|^2_{L^2}, \label{closed_theta}
\end{align}
where 
     \begin{align}
	\|W\|^2_{L^2}&=\|\bv_1-\bv_2\|^2_{L^2}
	=\frac{\tau}{\gamma} \|A^{-1} \mathbb{P} (\theta  \nabla (E^{-1}_{\tau}T_1)
	+T_2  \nabla (E^{-1}_{\tau}\theta) )   \|^2_{L^2}. \label{EST_W}
	\end{align}
Applying \eqref{2eta} to the right hand side of \eqref{EST_W} yields 
     \begin{align}
		&\|A^{-1} \mathbb{P} (\theta  \nabla (E^{-1}_{\tau}T_1)
	+T_2  \nabla (E^{-1}_{\tau}\theta) )   \|^2_{L^2} 
	=\left(\sup_{\psi\in D(A)}\frac{\int_{\Omega}[\mathbb{P} (\theta  \nabla (E^{-1}_{\tau} T_1)
	+T_2  \nabla (E^{-1}_{\tau} \theta)) ]\psi\, dx}{\|\psi\|_{H^2}}\right)^2\nonumber\\
	&\leq \left(\sup_{\psi\in D(A)}\frac{C ( \|\theta\|_{L^2}  \|\nabla (E^{-1}_{\tau} T_1)\|_{L^2}
	+\|T_2\|_{L^2} \|\nabla (E^{-1}_{\tau} \theta)\|_{L^2} ) \|\psi\|_{L^\infty}}{\|\psi\|_{H^2}}\right)^2 \label{2EST_W}\\
	&\leq  C ( \|\theta\|^2_{L^2} \frac{1}{2\kappa\tau}\|T_0\|^2_{L^2}
	+\|T_0\|^2_{L^2}\frac{1}{2\kappa\tau}\|\theta\|^2_{L^2} )
	\leq C \frac{1}{\kappa\tau}\| T_0\|^2_{L^2}\|\theta\|^2_{L^2},\label{3EST_W}
\end{align}
where from \eqref{2EST_W} to \eqref{3EST_W} we used  Agmon's inequality (cf.~\cite{temam1995navier}) that
\begin{align}
\| \psi\|_{L^{\infty}} \leq C\|\psi\|_{H^{d/2+\epsilon}},  d=2,3,   \quad \forall \epsilon>0. \label{agmon}
\end{align}
Thus \eqref{closed_theta} satisfies 
     \begin{align}
	&\frac{d\|\theta\|^2_{L^2}}{dt}
	\leq  \frac{C}{\gamma \kappa }  \| T_0\|^2_{L^\infty} \|T_0\|^2_{L^2}\|\theta\|^2_{L^2}. \label{2closed_theta}
\end{align}
Since $\|\theta_0\|_{L^2}=0$, by Gr\"{o}nwall inequality it is clear that  $\|\theta\|_{L^2}=0$. Therefore, the uniqueness of the solution is established.

To see \eqref{0Exp_TH1}--\eqref{1Exp_TH2}, we  first recall  the {\it a priori} estimates on $\|\nabla T\|_{L^2}$ and $\|\Delta T\|$ obtained in  \eqref{00EST_TH1} and Corollary \ref{cor1}.    
Using \eqref{00EST_TH1} together  with  Poncar\'{e} inequality 
 and \eqref{opt_vH2} we have     
            \begin{align}
         &     
       \frac{d \|\nabla T\|^2_{L^2}}{dt}
      +C \kappa \|\nabla T\|^2_{L^2}
      \leq
      \frac{d \|\nabla T\|^2_{L^2}}{dt}
      +\kappa \|\Delta T\|^2_{L^2} \nonumber\\
    &\qquad\quad \leq C \|\nabla \bv\|^2_{L^2}\|\nabla   T\|^2_{L^2}  
      \leq \frac{C\tau}{\kappa\gamma^2} \|T_0\|^2_{L^{\infty}}\| T_0\|^2_{L^2}\|\nabla   T\|^2_{L^2}, \label{Int_TH2}
\end{align}      
 which implies that if  $\frac{\tau}{\gamma^2}$      
    is chosen sufficiently small such that
   \[ C\kappa - \frac{C\tau}{\kappa \gamma^2} \|T_0\|^2_{L^{\infty}}\| T_0\|^2_{L^2}\ \geq \delta_0>0, \]
   then \eqref{0Exp_TH1} holds. Moreover, from \eqref{Int_TH2} we can easily  verify \eqref{1Exp_TH1}.

   In addition, in light of \eqref{3EST_TH2} we also have 
   \begin{align*}
&\frac{d\|\Delta  T\|^2_{L^2}}{dt}+\kappa  \|\Delta T\|^2_{L^2}
 \leq \frac{d\|\Delta  T\|^2_{L^2}}{dt}+C \kappa \|\nabla (- \Delta  T)\|^2_{L^2}  \\
&\qquad \leq C\|\bv\|^2_{H^2}\| \Delta  T\|^2_{L^2}
  \leq  \frac{C\tau}{\kappa\gamma^2} \|T_0\|^2_{L^{\infty}}\| T_0\|^2_{L^2}\|\Delta    T\|^2_{L^2}. 
\end{align*}
Analogously, if $\frac{\tau}{\gamma^2}$       
    is  sufficiently small  such that
\[ C\kappa -  \frac{C\tau}{\kappa\gamma^2} \|T_0\|^2_{L^{\infty}}\| T_0\|^2_{L^2} \geq \delta_1>0, \]
then  \eqref{0Exp_TH2}--\eqref{1Exp_TH2} hold.  
Consequently, 
 \begin{align*}
	\|\frac{\partial T}{\partial t}\|_{L^2}& \leq  \kappa \|\Delta T\|_{L^2}+\|\bv \cdot \nabla T \|_{L^2}
	 \leq \kappa e^{-\delta_1t}\|\Delta T_0\|^2_{L^2} 
	+ \frac{C\tau^{1/2}}{ \kappa^{1/2}\gamma} \|T_0\|_{L^{\infty}}\| T_0\|_{L^2}e^{-\delta_0t}\|\nabla  T_0\|^2_{L^2},
\end{align*}
which yields \eqref{Exp_Tt}. This completes the proof.
\end{proof}

\begin{remark}\label{mix}

Note that the estimates in \eqref{0Exp_TH1}--\eqref{Exp_Tt} only provide upper bounds for the decay rates of the temperature evolution, which also hold when $\tau$ is set to be zero, i.e., $\bv=0$ or no advection. However, our numerical results indicate that the feedback law always performs better than ``do nothing" with properly chosen  parameters.  On the other hand, if using the negative Sobolev norm, or equivalently, the dual norm $(H^{s}(\Omega))'$, for any  $s>0$,  as the ``mix-norm" for  quantifying homogenization  of a scalar field, which are sensitive for   both diffusion and pure advection  effects (cf.~\cite{  hu2018boundary, hu2018boundarySICON, hu2018approximating, hu2019approximating, lin2011optimal, lunasin2012optimal, mathew2005multiscale, thiffeault2012using}), we realize that  the decay rate of $\|T\|_{(H^{1}(\Omega))'}$ is indeed enhanced by the nonlinear feedback law. 
To see this, taking the inner product of \eqref{2closed_loop} with  $\eta=E^{-1}_{\tau} T$ defined in \eqref{eta} and using \eqref{2EST_tri},   we obtain
            \begin{align*}
             &  \frac{1}{2}   \frac{d  (\|\eta \|^2_{L^2}+\kappa \tau \|\nabla \eta\|_{L^2})}{d  t} 
             +\kappa\|\nabla \eta\|^2_{L^2}+\kappa^2\tau \|\Delta \eta\|^2_{L^2}
     = \frac{\tau}{\gamma} (A^{-1}\mathbb{P} (T \nabla (E^{-1}_{\tau}T  ) )\cdot \nabla T, E^{-1}_{\tau} T)\\
&\qquad=-\frac{\tau}{\gamma} (A^{-1}\mathbb{P} (T \nabla (E^{-1}_{\tau} T  ) ), \mathbb{P} (T\nabla (E^{-1}_{\tau} T)))
=-\frac{\tau}{\gamma} \|A^{-1/2}\mathbb{P} (T \nabla (E^{-1}_{\tau} T  ) )\|^2_{L^2},
            \end{align*}
 and therefore,
        \begin{align}
             &  \frac{1}{2}   \frac{d  (\|\eta \|^2_{L^2}+\kappa \tau \|\nabla \eta\|^2_{L^2})}{d  t} 
             +\kappa\|\nabla \eta\|^2_{L^2}+\kappa^2\tau \|\Delta \eta\|^2_{L^2}
+\frac{\tau}{\gamma} \|A^{-1/2}\mathbb{P} (T \nabla (E^{-1}_{\tau} T  ) )\|^2_{L^2}=0. \label{EST_eta}
            \end{align}
  Similarly,   if $\bv=0$, let $\eta=(I-\Delta )T$ in $\Omega$ with $\frac{\partial \eta}{\partial n}|_{\Gamma}=0$. Then 
      \begin{align}
             &  \frac{1}{2}   \frac{d  (\|\eta \|^2_{L^2}+ \|\nabla \eta\|^2_{L^2})}{d  t} 
             +\kappa\|\nabla \eta\|^2_{L^2}+\kappa \|\Delta \eta\|^2_{L^2}=0.\label{decay_v0}
            \end{align}
Since  $\|T\|_{(H^{1}(\Omega))'}$   is equivalent to $\| \eta\|_{H^1}$ for a fixed $\tau>0$, compared to \eqref{decay_v0} it is clear that  the decay rate of $\| \eta\|_{H^1} $  is accelerated  in \eqref{EST_eta} with the presence of the positive nonlinear  term by setting $\tau=\frac{1}{\kappa}$. However, due to the complexity of the nonlinearity together with the Leray projector, it is rather challenging to have a  thorough  understanding of   this nonlinear mechanism in enhancing convection-cooling or the homogenization process. 
\end{remark}

\section{Numerical examples} \label{Num}
In this section we present some numerical examples to validate the performance of our control designs.
We will iteratively solve the nonlinear optimality system in Theorem \ref{thm1}  via the standard Picard iteration (with the linearization of the velocity filed $\bv$):
\begin{eqnarray}\label{PicardIter}
	\begin{cases}
		\frac{\partial T^{(k+1)}}{\partial t}=\kappa \Delta T^{(k+1)}-\bv^{(k)} \cdot \nabla T^{(k+1)}, \quad \frac{\partial T^{(k+1)}}{\partial n}|_{\Gamma} = 0, \quad T^{(k+1)}(0)=T_0\\
		-\frac{\partial q^{(k+1)}}{\partial t}=\kappa \Delta q^{(k+1)}+{\bv^{(k)}}  \cdot \nabla   q^{(k+1)}  + \beta D^*DT^{(k+1)}, \quad	\frac{\partial q^{(k+1)}}{\partial n}|_{\Gamma} = 0, \quad q^{(k+1)}(T)=\alpha D^*DT^{(k+1)}(t_f),\\
		-\gamma\Delta {\bv^{(k+1)}}+\nabla p^{(k+1)} = q^{(k+1)} \nabla  T^{(k+1)},\quad \nabla\cdot \bv^{(k+1)} = 0, \quad \bv^{(k+1)}|_{\Gamma} = 0,
	\end{cases}
\end{eqnarray}
where  $\bv^{(k)}$ denotes the velocity field at $k$-th Picard iteration with  $\bv^{(0)}$ being a given zero initial guess.
In implementation of the Picard iteration, we will use a uniform mesh with center finite difference scheme in space (with a step size $\Delta x=1/N_x$ and $\Delta y=1/N_y$ in  $x$ and $y$ direction respectively) and semi-implicit Euler scheme in time (with a step size $\Delta t=t_f/N_t$),
where the Stokes equation is discretized by the MAC scheme. 
Clearly, the Picard iteration is expensive since it consists of forward marching in $T$, backward marching in $q$, and solving $N_t$ Stokes equations over all time points.
Define a nonlinear iterative mapping $G: \bv^{(k)} \to \bv^{(k+1)}$.
If the above Picard iteration is assumed to converge  in certain norm under suitable assumptions (e.g. $\gamma$ is not too small), that is,
$\lim_{k\to\infty} \bv^{(k)}=\bv$ exists,
then the Picard iteration essentially finds
a fixed point $\bv$ of the nonlinear mapping $G$, i.e., 
$
\bv=G(\bv).
$
Since our problem is non-convex, such a fixed point in general may not be unique, and which fixed point the Picard iteration may (locally) converge to depends highly on the  initial guess and the  numerical implementation method (such as the used discretization schemes).
For faster convergence, we will interpolate the coarse mesh  solution as a reasonably good initial guess, where the mesh sizes is doubled in refinement starting with $(N_x,N_y,N_t)=(10,10,10)$.
If convergent, the convergence rate of the Picard iteration can be very slow, depending on the given model parameters.
Anderson acceleration (AA) technique \cite{WalkerAA2011} can be employed to significantly speed up the convergence of the Picard iteration. 
Our numerical results show that such a Picard iteration based on AA technique converges very fast, and its implementation is much simpler than the standard Newton method that requires to solve a large-scale Jacobian system at each iteration. We mention that the local convergence radius of Newton iterations is usually much smaller than that of the Picard iterations, which however can be combined with the Picard iterations. 
More robust nonlinear solvers are desirable for solving the optimality system, which will be part of our future work.

The nonlinear feedback control is more straightforward to compute.
We solve the closed-loop continuous nonlinear parabolic PDEs by a standard semi-implicit Euler scheme in time (with the same step size $\Delta t$), where the nonlinear convection term (desired control) involving a Stokes equation is treated explicitly for better computational efficiency and the same MAC scheme is employed for the underlying Stokes equations.  
The simulation of close-loop feedback control system is expected to be more efficient than the open-loop optimal control whenever the number of Picard iterations for convergence is not small.

All numerical simulations are implemented using MATLAB on a laptop PC with Intel(R) Core(TM) i7-7700HQ CPU@2.80GHz CPU and 32GB RAM, where CPU times  (in seconds) are estimated by the timing functions \texttt{tic/toc}.
The stopping tolerance for the AA-Picard iteration (with 5 memory iterations) is $10^{-5}$.
We choose  the spatial domain $\Omega=(0,1)^2$, the diffusion coefficient $\kappa=0.05$,  the penalty parameter $\gamma=0.025$, and $t_f=1$ in all tested examples. For the feedback control system, we will test a few selected parameter  {$\tau\in\{0.25,0.5,0.75,1\}\subset (0,t_f]$ and then plot the best choice for an illustrative comparison}.
For a fixed $\gamma$, a very small $\tau$ gives little or insignificant control effects, while a very large $\tau$ leads to stronger control that may greatly increase objective functionals.
The optimal choice of parameter $\tau$ seems to be non-trivial and it highly depends on the penalty parameter $\gamma$ and the nonlinearity. 

For the purpose of direct comparison, we write the objective functional into three terms:
\begin{align*}
	J(\bv)&=\underbrace{\frac{\alpha}{2}\|T(x, t_f)-\langle T(x,t_f)\rangle\|^2_{L^2}}_{=:J_\alpha}+
	\underbrace{\frac{\beta}{2}\int^{t_f}_0\|T-\langle T\rangle\|^2_{L^2}\, dt}_{=:J_\beta}
	+\underbrace{\frac{\gamma}{2}\|\bv\|^2_{U_{\text{ad}}}}_{=:J_\gamma},  
\end{align*}
where $J_\alpha\equiv 0$ if choosing $\alpha=0$  
and $J_\gamma\equiv 0$ if there is no control ($\bv=0$).
For a fair comparison, we will only consider the case with $\alpha=0$ in the following examples.
We highlight that the nonlinear feedback control derived in the previous section  is sub-optimal and  its performance may be problem  dependent and also sensitive to the choice of slicing parameter $\tau$,  the control weight $\gamma$, as well as the initial temperature distribution.
Our  current numerical schemes  may  only find local minimizers since a global minimizer for such a non-convex optimization problem is in general difficult (or NP-hard) to find, which requires global optimization techniques that are beyond our reach.
\subsection{Example 1}
The first example uses the smooth initial condition with an oval-shaped bump given by 
\begin{align*}
  T_0(x,y)&=10\left(0.5+\frac{1}{\pi}\arctan\left(10(1-32(x-0.25)^2-16(y-0.25)^2)\right)\right),
\end{align*}
where the initial heated region is located within an ellipse centered at $(0.25,0.25)$.
We compare the control outcomes of three different scenarios: no control,  optimal control and feedback control (with different choices of $\tau$).
Table \ref{Ex1Tab1} reports the attained different objective functionals and control measurements, where `Iter' denotes the number of Picard iterations used for solving the nonlinear optimality system, and the two control measurements are computed as the maximum over $[0, t_f]$ by
$$\VERT\nabla\cdot \bv\VERT_\infty:=\max_{0\le t\le t_f} \|\nabla\cdot \bv(t)\|_{L^2(\Omega)},\qquad \VERT\bv\VERT_\infty:=\max_{0\le t\le t_f} \|\bv(t)\|_{L^2(\Omega)}.$$
We mention that the divergence-free condition $\nabla\cdot \bv=0$ holds only approximately due to discretization errors.
As expected, the computation of feedback control costs much less CPU times than the optimal control (with over 8 million decision variables for velocity field with a $160\times 160\times 160$ mesh).
Figures \ref{fig_ex1_A}--\ref{fig_ex1_B} show the decay of $\|DT(t)\|$ and $\|\bv(t)\|$ and the snapshots of temperature distribution and control velocity field at different time points, respectively.
 {The exponential decay of  $\|DT(t)\|$ with no control is observed which  clearly verifies our analysis (see Remark \ref{prop_aveT}), and the decay rates via controlled advection  are anticipated to be  faster}. 
For this particular example, the feedback control (with the choice $\tau=0.75$) and the optimal  control   provide about 26.2\% and 28.5\% reduction, respectively,  in the objective functionals compared to the case  with no control.  Moreover, both controls (based on very different numerical implementations) generate   very similar dynamical patterns as shown  in Figures \ref{fig_ex1_A}-\ref{fig_ex1_B}. This example also suggests  that the  feedback control law  can be as effective as the optimal control.
Nevertheless, we acknowledge that the optimal choice of parameter $\tau$ is a non-trivial task, which merits further analysis. Numerically we do observe the best choice of $\tau$ lies between 0.5 and 1.
\begin{table}[H]
	\centering
	\caption{Control performance comparison of Example 1 with Neumann BC ($\alpha=0,\beta=1,\gamma=0.025$)}
	
	\begin{tabular}{|c||c|c|c|c|c|c|c|c|c|c}
		\hline
		Control& $(N_x,N_y,N_t)$&  $J(\bv)$&  $J_{\beta}$& $J_{\gamma}$& 
		$\VERT\nabla\cdot \bv\VERT_\infty$& $\VERT\bv\VERT_\infty$& Iter& CPU \\
		\hline 
		\multirow{1}{*}{None ($\bv=0$)} 
&(160,160,160)	 &1.559	   &1.559	 &0.000	 &0.000	 &0.00	 &--	 &15.0 \\ 
\hline
\multirow{1}{*}{Optimal}  
&(160,160,160)	 &\textbf{1.114}	  	 &0.852	 &0.263	 &0.006	 &1.89	 &21	 &766.5 \\

\hline
		\multirow{1}{*}{Feedback ($\tau=0.25$)}  
&(160,160,160)	 &1.380	  	 &1.352	 &0.028	 &0.010	 &0.57	 &--	 &230.2 \\ 	
		\multirow{1}{*}{Feedback ($ {\tau=0.5}$)}   
&(160,160,160)	 &1.170	  	 &1.011	 &0.159	 &0.014	 &1.22	 &--	 &228.9 \\  
 
		\multirow{1}{*}{Feedback ($\tau=0.75$)}  
&(160,160,160)	 &\textbf{1.150}	 &0.838	 &0.312	 &0.017	 &1.96	 &--	 &229.4 \\ 	

		\multirow{1}{*}{Feedback ($\tau=1.0$)}  
&(160,160,160)	 &1.207	  	 &0.757	 &0.449	 &0.020	 &2.62	 &--	 &229.4 \\ 	


		\hline  
		
	\end{tabular}
	\label{Ex1Tab1}
\end{table}

\begin{figure}[H]
	\centering 
	\includegraphics[width=1\textwidth]{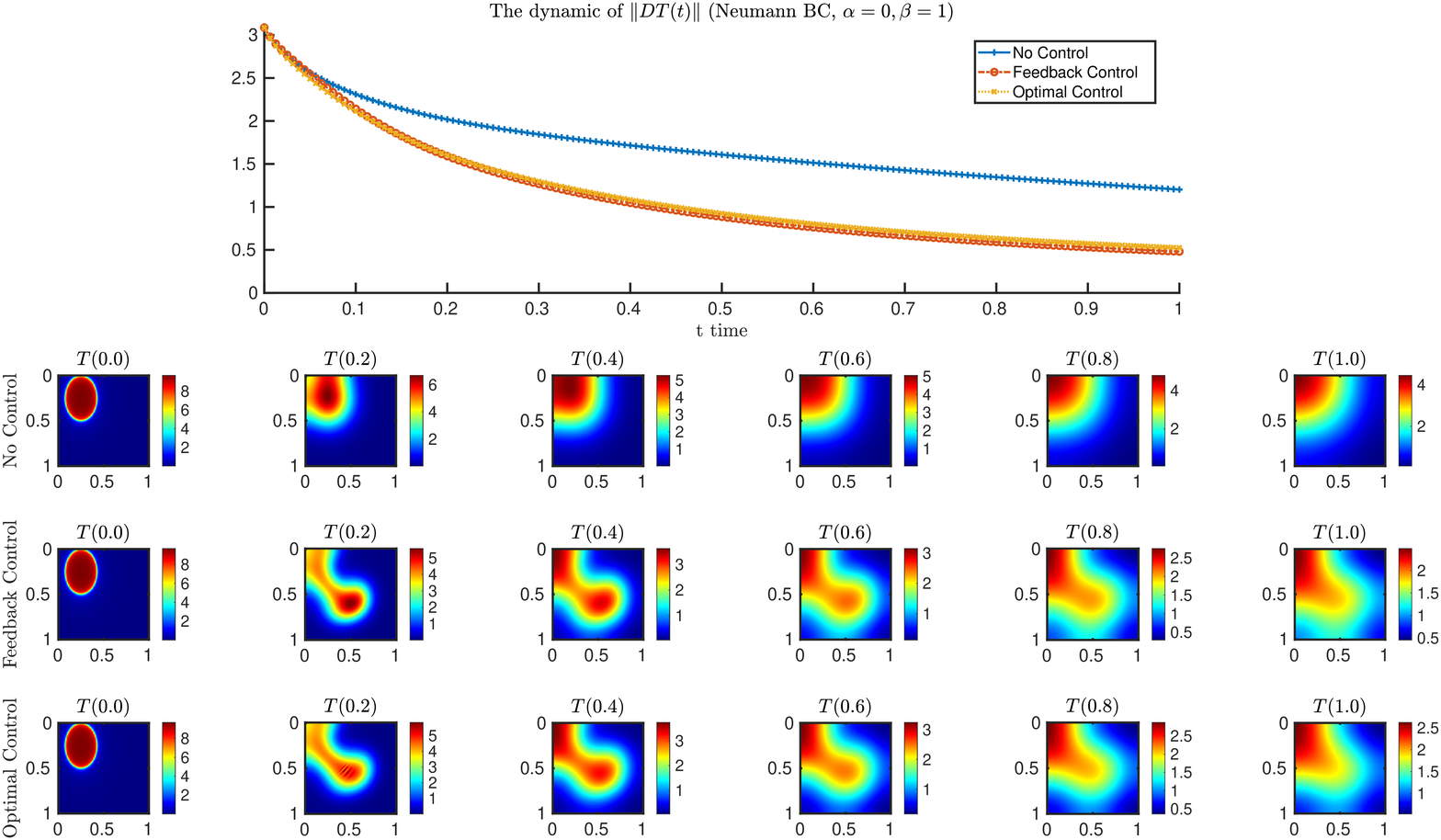}
	\caption{The  snapshots of state $T(t)$ at different time points for Example 1 ($t_f=1,\tau=0.75$,$\alpha=0,\beta=1$).} \label{fig_ex1_A}
\end{figure}

\begin{figure}[H]
	\centering 
	\includegraphics[width=1\textwidth]{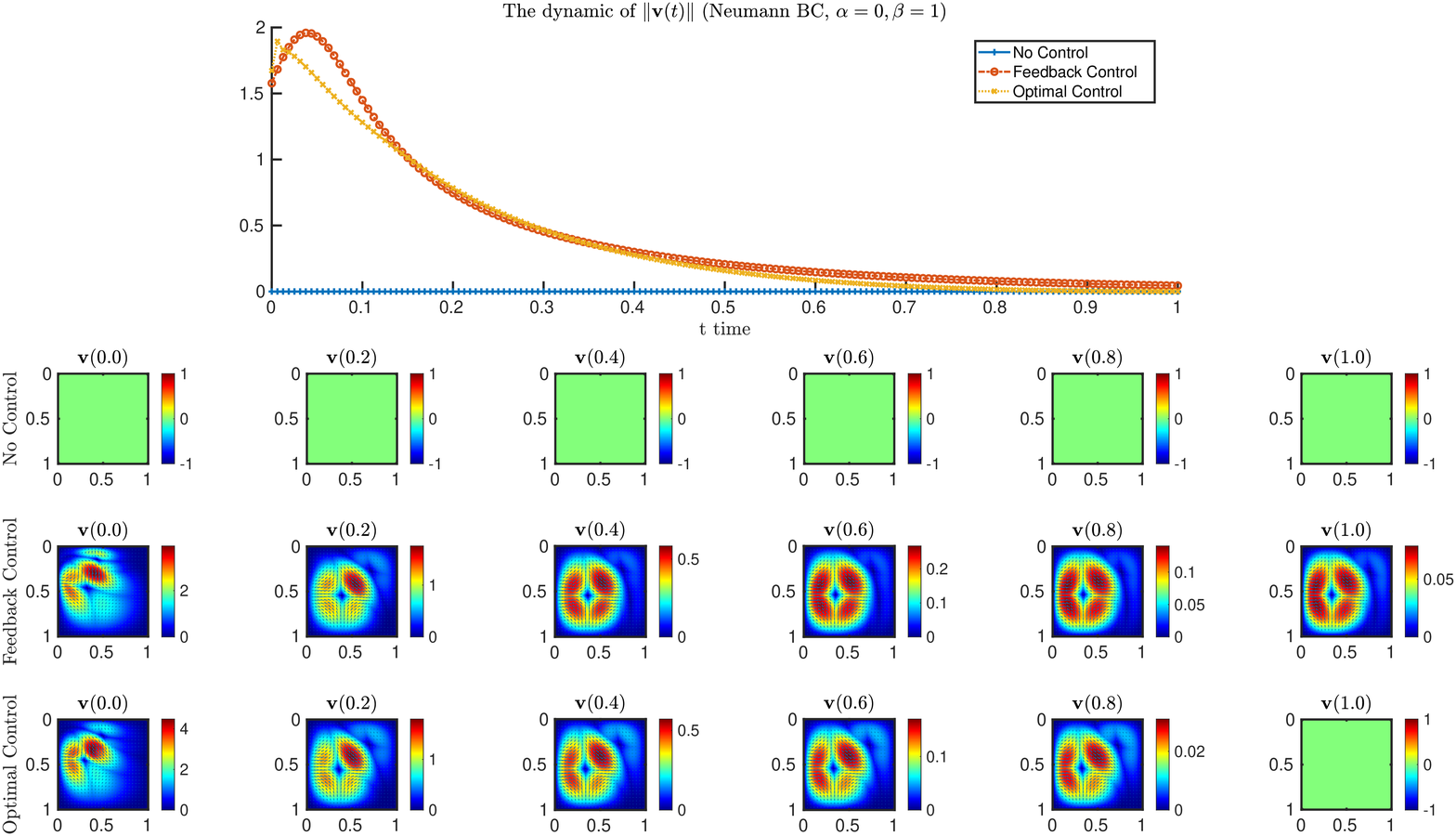}
	\caption{The snapshots of control $\bv(t)$ at different time points  for Example 1 ($t_f=1,\tau=0.75$,$\alpha=0,\beta=1$).} \label{fig_ex1_B}
\end{figure}
 
\subsection{Example 2}
The second example considers the smooth initial condition with two oval-shaped bumps defined by 
\begin{align*}
	T_0(x,y)&=10\left(0.5+\frac{1}{\pi}\arctan\left(10(1-32(x-0.25)^2-16(y-0.25)^2)\right)\right) \\
	&+10\left(0.5+\frac{1}{\pi}\arctan\left(10(1-32(x-0.75)^2-16(y-0.25)^2)\right)\right),
\end{align*}
where the two heated regions are located within two ellipses centered at $(0.25,0.25)$ and $(0.75,0.25)$.
Table \ref{Ex2Tab1} reports the attained different objective functionals and control measurements. Figures \ref{fig_ex2_A}--\ref{fig_ex2_B} present the decay  of $\|DT(t)\|$ and $\|\bv(t)\|$  and the snapshots of temperature distribution and control velocity field at different time points, respectively.  
Similar to Example 1, the feedback control (with the choice $\tau=0.75$) and optimal  control  provide   about 26.2\% and 29.4\% reduction, respectively, in the objective functionals compared to  the case  with no control.
However, Figure \ref{fig_ex2_A} demonstrates  that different controls may lead to  very different evolution of temperature distribution.

\begin{table}[H]
	\centering
	\caption{Control performance comparison of Example 2 with Neumann BC ($\alpha=0,\beta=1,\gamma=0.025$)}
	
	\begin{tabular}{|c||c|c|c|c|c|c|c|c|c|c}
		\hline
		Control& $(N_x,N_y,N_t)$&  $J(\bv)$&   $J_{\beta}$& $J_{\gamma}$& 
		$\VERT\nabla\cdot \bv\VERT_\infty$& $\VERT\bv\VERT_\infty$& Iter& CPU \\
		\hline 
		\multirow{1}{*}{None ($\bv=0$)} 
		 &(160,160,160)	 &2.296	  	 &2.296	 &0.000	 &0.000	 &0.00	 &--	 &13.2 \\ 
		\hline
		\multirow{1}{*}{Optimal}  
		 &(160,160,160)	 &\textbf{1.622}	  	 &1.163	 &0.460	 &0.013	 &1.95	 &14	 &528.6 \\
		
		\hline
		\multirow{1}{*}{Feedback ($\tau=0.25$)}  
		&(160,160,160)	 &2.030	  	 &1.990	 &0.040	 &0.013	 &0.61	 &--	 &223.3 \\  	
		\multirow{1}{*}{Feedback ($ {\tau=0.5}$)}   
	 &(160,160,160)	 &1.766	  	 &1.557	 &0.209	 &0.017	 &1.32	 &--	 &223.5 \\ 
		
		\multirow{1}{*}{Feedback ($\tau=0.75$)}  
	 &(160,160,160)	 &\textbf{1.695}	  	 &1.269	 &0.427	 &0.020	 &1.93	 &--	 &225.4 \\ 
		
		\multirow{1}{*}{Feedback ($\tau=1.0$)}  
	 &(160,160,160)	 &1.730	  	 &1.091	 &0.638	 &0.023	 &2.46	 &--	 &224.5 \\ 
 
		\hline  
		
	\end{tabular}
	\label{Ex2Tab1}
\end{table}

\begin{figure}[H]
	\centering 
	\includegraphics[width=1\textwidth]{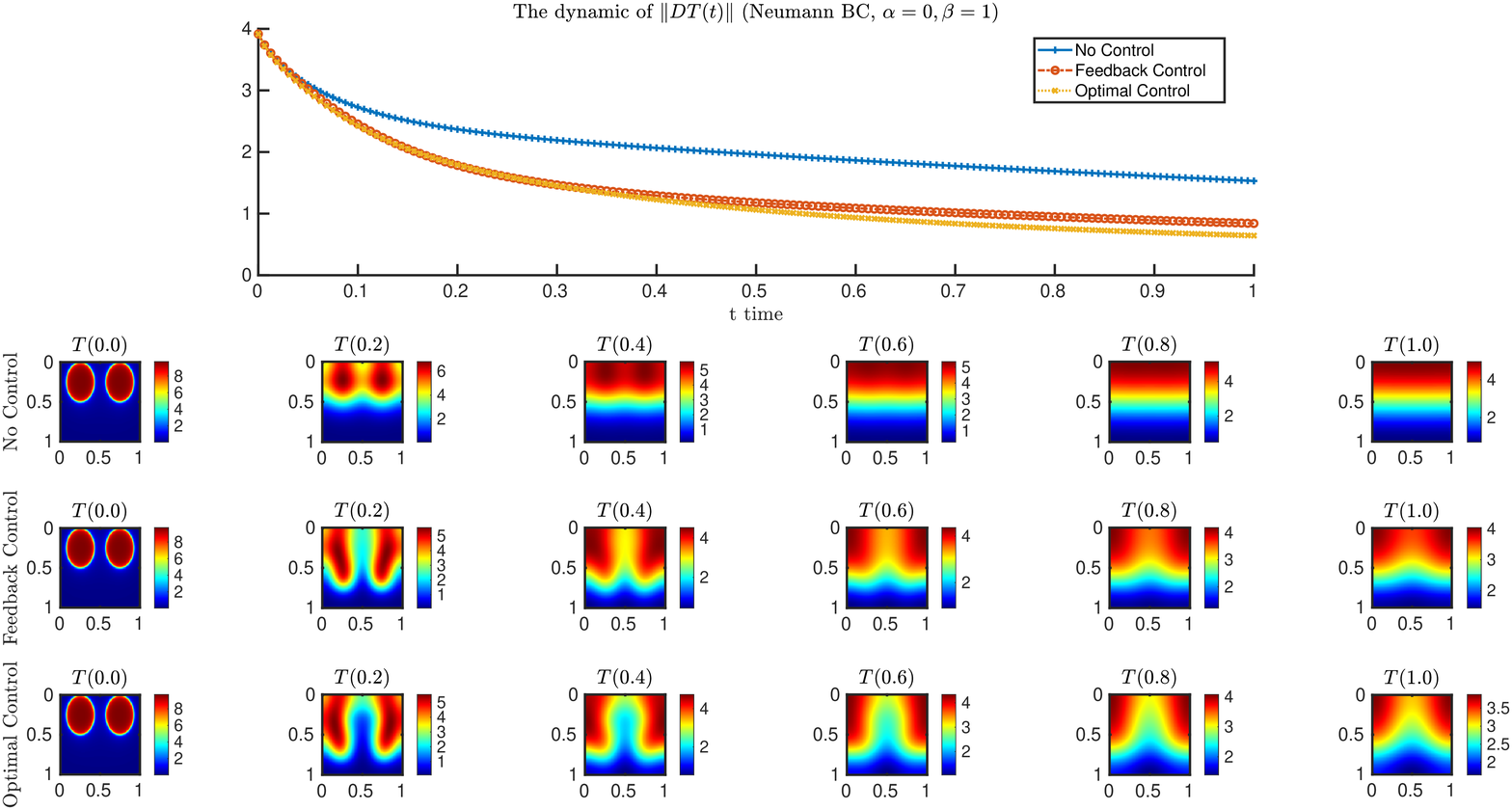}
	\caption{The  snapshots of state $T(t)$ at different time points for Example 2 ($t_f=1,\tau=0.75$,$\alpha=0,\beta=1$).} \label{fig_ex2_A}
\end{figure}

\begin{figure}[H]
	\centering 
	\includegraphics[width=1\textwidth]{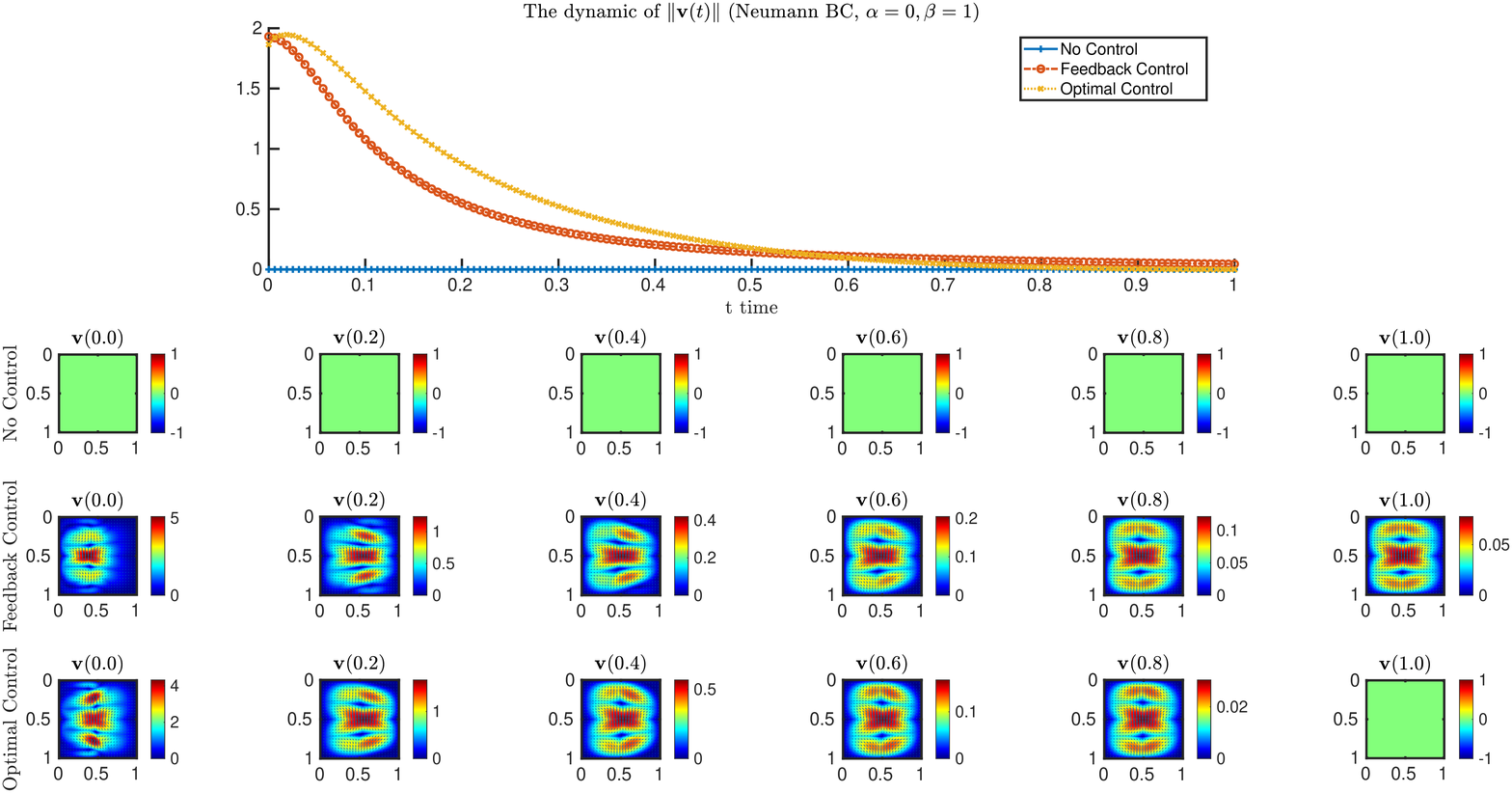}
	\caption{The snapshots of control $\bv(t)$ at different time points  for Example 2 ($t_f=1,\tau=0.75$,$\alpha=0,\beta=1$).} \label{fig_ex2_B}
\end{figure}

 \subsection{Example 3}
 For the sake of numerical test, the third example examines  the initial condition with two squared bumps given by 
 \[
 T_0(x,y)=10\times \mathds{1}_{S} ,
 \] 
 with $S=[0,0.5)^2\cup (0.5,1]^2$ and $\mathds{1}$ denotes the indicator function. 
 In this case, the initial condition $T_0$ is indeed discontinuous, 
 but it will be quickly smoothed out due to diffusion.
 Table \ref{Ex3Tab1} reports the attained different objective functionals and control measurements. Figures \ref{fig_ex3_A}--\ref{fig_ex3_B} present  the decay of $\|DT(t)\|$ and $\|\bv(t)\|$ and the snapshots of temperature distribution and control velocity field at different time points, respectively.
 Compared with no control, the optimal control provides 22.8\% reduction in  $J(\bv)$, while the feedback control (with $\tau=1$) attains only 7.7\% reduction in $J(\bv)$.
 The controlled dynamics demonstrate quite different pattern during the early stage.
 Again, the computation of optimal control costs about three times longer CPU time than the feedback control. 
 This example shows that  the sub-optimal feedback control may be far away from being optimal.
 Similar results can be obtained with the corresponding smoothed initial condition (e.g. use smooth rounded squares as heated source).
 \begin{table}[H]
 	\centering
 	\caption{Control performance comparison of Example 3 with Neumann BC ($\alpha=0,\beta=1,\gamma=0.025$)}
 	
 	\begin{tabular}{|c||c|c|c|c|c|c|c|c|c|c}
 		\hline
 		Control& $(N_x,N_y,N_t)$&  $J(\bv)$&  $J_{\beta}$& $J_{\gamma}$& 
 		$\VERT\nabla\cdot \bv\VERT_\infty$& $\VERT\bv\VERT_\infty$& Iter& CPU \\
 		\hline 
 		\multirow{1}{*}{None ($\bv=0$)} 
 &(160,160,160)	 &3.950	   &3.950	 &0.000	 &0.000	 &0.00	 &--	 &12.8 \\ 
 		\hline
 		\multirow{1}{*}{Optimal}  
&(160,160,160)	 &\textbf{3.049}	  &2.144	 &0.905	 &0.019	 &2.41	 &19	 &741.3 \\
 		\hline
 		\multirow{1}{*}{Feedback ($\tau=0.25$)}  
 	&	(160,160,160)	 &3.942	  	 &3.941	 &0.002	 &0.000	 &0.07	 &--	 &231.2 \\ 
 		 
 		\multirow{1}{*}{Feedback ({$\tau=0.5$})}  
 	&	 (160,160,160)	 &3.919	  	 &3.901	 &0.018	 &0.001	 &0.17	 &--	 &229.6 \\ 
 		
 		\multirow{1}{*}{Feedback ($\tau=0.75$)}  
 	&	 (160,160,160)	 &3.805	  	 &3.602	 &0.203	 &0.002	 &0.66	 &--	 &230.0 \\ 
 		\multirow{1}{*}{Feedback ($\tau=1.0$)}  
 	&	 (160,160,160)	 &\textbf{3.647}	  	 &3.090	 &0.557	 &0.004	 &1.38	 &--	 &230.0 \\  
 \hline
 	 	\multirow{1}{*}{Feedback ($\tau=1.25$)}  
 	&(160,160,160)	 &\underline{3.599}	  	 &2.750	 &0.849	 &0.007	 &1.99	 &--	 &228.5 \\ 
 	  	 	\multirow{1}{*}{Feedback ($\tau=1.5$)}  
 	&(160,160,160)	 &3.617	  	 &2.536	 &1.081	 &0.009	 &2.47	 &--	 &228.6 \\ 
 		\multirow{1}{*}{Feedback ($\tau=1.75$)} 
 	 	&(160,160,160)	 &3.660	   &2.391	 &1.269	 &0.011	 &2.87	 &--	 &228.7 \\
 		\hline  
 		
 	\end{tabular}
 	\label{Ex3Tab1}
 \end{table}

 \begin{figure}[H]
 	\centering 
 	\includegraphics[width=1\textwidth]{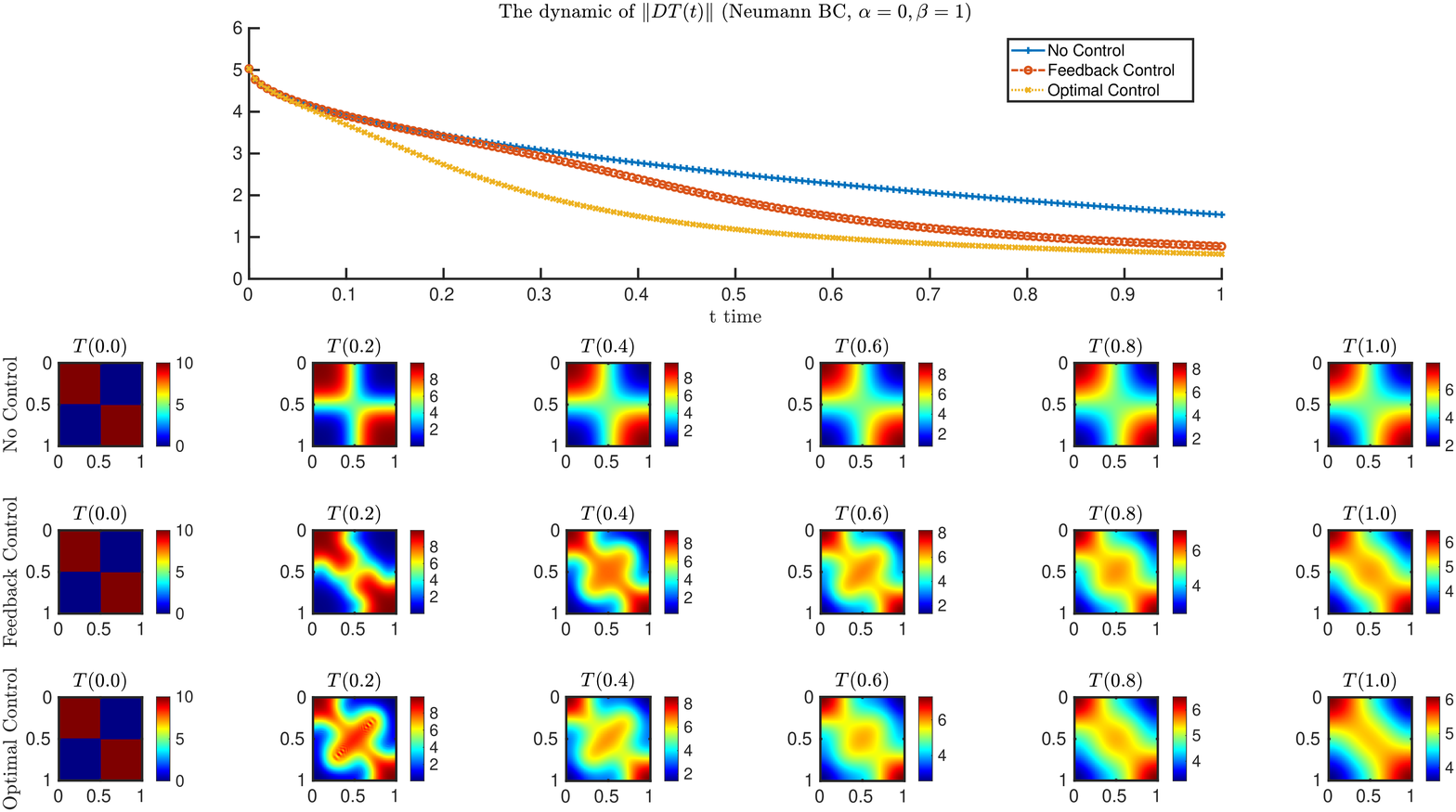}
 	\caption{The  snapshots of state $T(t)$ at different time points for Example 3 ($t_f=1,\tau=1$,$\alpha=0,\beta=1$).} \label{fig_ex3_A}
 \end{figure}
 
 \begin{figure}[H]
 	\centering 
 	\includegraphics[width=1\textwidth]{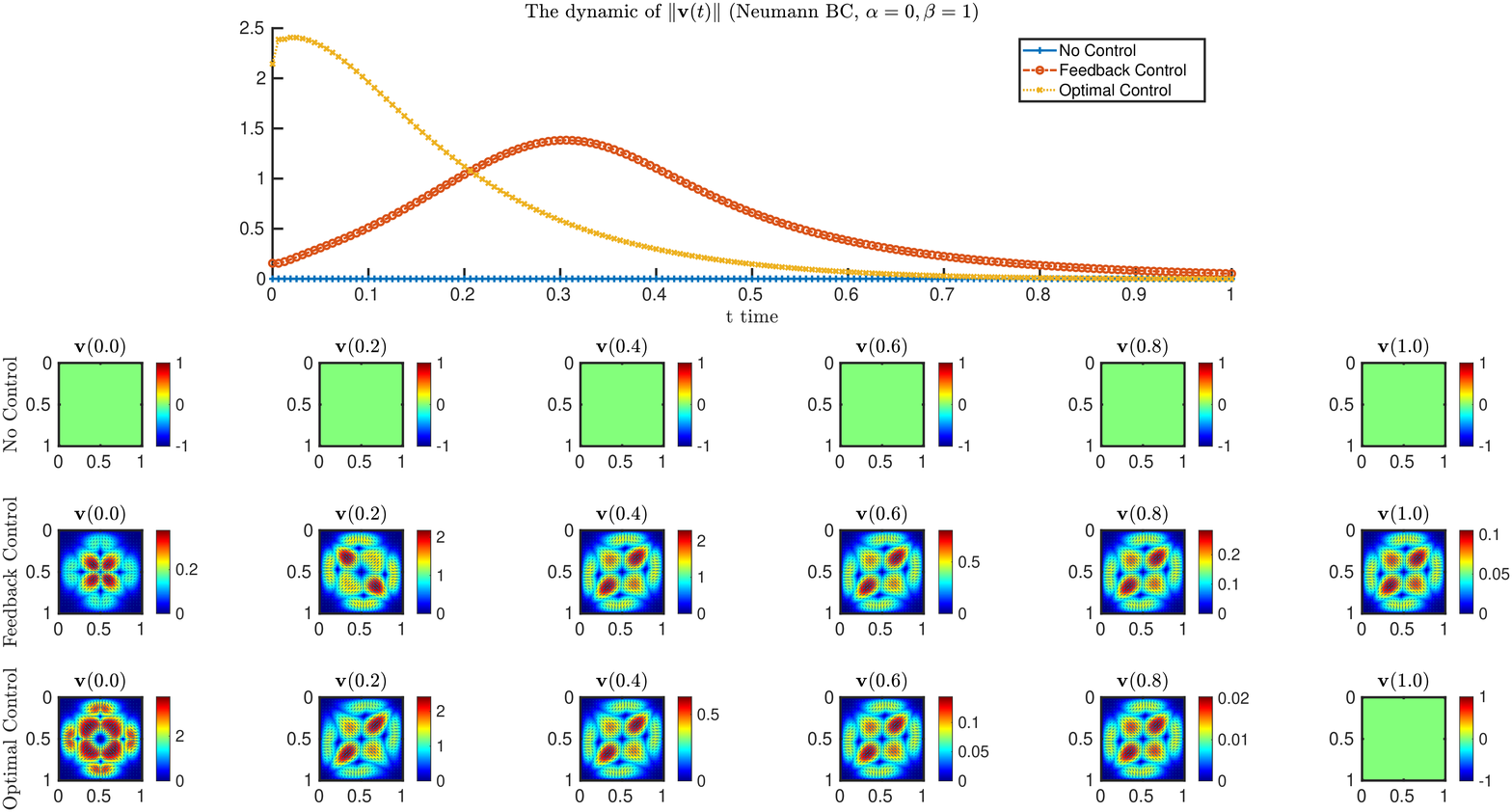}
 	\caption{The snapshots of control $\bv(t)$ at different time points  for Example 3 ($t_f=1,\tau=1$,$\alpha=0,\beta=1$).} \label{fig_ex3_B}
 \end{figure}

 {
  	To illustrate how the performance of feedback control depends on the key parameter $\tau\ge 0$,  we plot in Figure  \ref{fig_ex3_C} the values of
  	$J(\bv)$ as a function of $\tau\in[0,2]$.
  	It shows the best choice of $\tau$ lies in the open interval $(1.2,1.4)$. This can also be seen from the last three rows in Table \ref{Ex3Tab1}, where the feedback control with $\tau=1.25$ provides a slightly smaller $J(\bv)$ than with $\tau=1.0$.
  	Based on the previous examples, the best value of $\tau>0$ seems to be problem dependent, which may not necessarily be less than $t_f=1$, although it was originated as a step size. 
  }
  \begin{figure}[H]
  	\centering 
  	\includegraphics[width=0.8\textwidth]{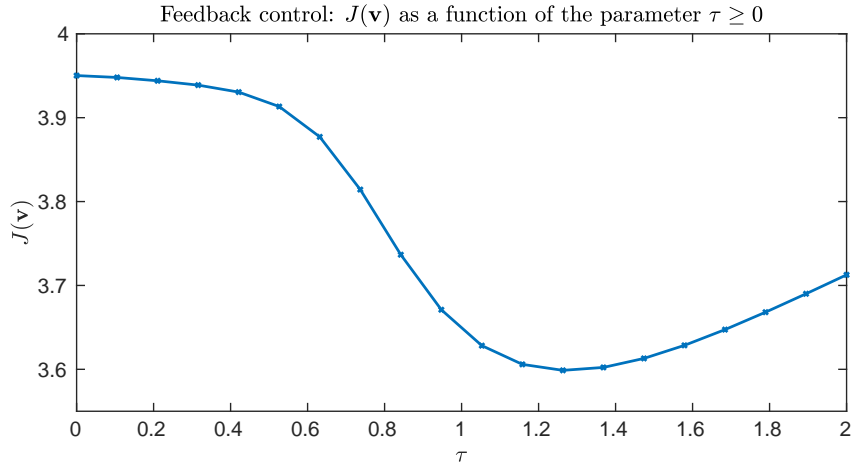}
  	\caption{Feedback control: the value of $J(\bv)$ as a function of the parameter $\tau\in[0,2]$ for Example 3 ($t_f=1$,$\alpha=0,\beta=1$).} \label{fig_ex3_C}
  \end{figure}
  
\section{Conclusions}

In the current work, we have discussed  both optimal and feedback controls for convection-cooling  via incompressible fluid flows.  First and second necessary optimality conditions were derived for solving and characterizing the optimal control. 
Motivated by the method   of instantaneous  control,  we investigated the idea of directly  constructing  the  feedback  laws by making use of the optimality conditions together with  numerical discretization schemes. Our numerical experiments demonstrated the effectiveness of the different control designs. In particular, the sub-optimal feedback control demonstrates  comparable performances as the optimal control in some cases.  However, there is no rigorous proof for justifying the optimality of the  feedback law. Understanding how exactly the mechanism   of the nonlinear feedback law plays in the enhancement of   convection-cooling or homogenization of a general scalar field, especially, its relation to  the diffusivity $\kappa$, the parameter $\tau$ as well as  the control weight $\gamma$, requires a more in-depth analysis. The aforementioned issues  will be investigated in our future work. 
\section*{Appendix}\label{app}

%
\noindent\textbf{Proof of Corollary  \ref{cor1}.}
\begin{proof}
First, with the help of Lemma \ref{lem1}, \eqref{EST_qinfty}, and the optimality   condition \eqref{opt_v} we have
\begin{align}
&\int^{t_f}_0\|\bv\|^2_{H^2} \,dt\leq C\int^{t_f}_0\|q\nabla T\|^2_{L^2}\,dt
\leq C\sup_{t\in[0, t_f]}\|q\|^2_{L^\infty}\int^{t_f}_0\|\nabla T\|^2_{L^2}\,dt\leq C(T_0, t_f). \label{EST_opt_v}
\end{align}

Moreover, by \eqref{0EST_TH1} and \eqref{EST_opt_v} we have
\begin{align}
&\sup_{t\in [0, t_f]}\|\bv\|_{H^2} \leq C\sup_{t\in [0, t_f]}\|q\nabla T\|_{L^2}
\leq C\sup_{t\in[0, t_f]}\|q\|_{L^\infty}\sup_{t\in [0, t_f]} \|\nabla T\|_{L^2}\leq C(T_0, t_f). \label{2EST_opt_v}
\end{align}
To  obtain a higher regularity of $T$,  we take the inner product of \eqref{opt_T} with $(-\Delta)^2T$ and get 
\begin{align*}
&\frac{1}{2}\frac{d\|\Delta  T\|^2_{L^2}}{dt}+\kappa \|\nabla (-\Delta T)\|^2_{L^2} 
=-(\bv\cdot \nabla T, (-\Delta )^2T)
=(\nabla (\bv\cdot \nabla T), \nabla ((-\Delta )T))\nonumber\\
&\qquad\quad\leq C \|\nabla (\bv\cdot \nabla T)\|_{L^2}\|\nabla ((-\Delta)T\|_{L^2}
\leq  C (\|\nabla \bv\cdot  \nabla T\|_{L^2}+\| \bv\cdot  \nabla (\nabla T)\|_{L^2})\|\nabla ((-\Delta)T\|_{L^2}\nonumber\\
&\qquad\quad \leq C (\|\nabla \bv\|^2_{H^1}\|\Delta T\|^2_{L^2}+\| \bv\|^2_{L^\infty}\| \Delta  T\|^2_{L^2})
+\frac{\kappa}{2}\|\nabla ((-\Delta)T\|^2_{L^2}. 
\end{align*}
This 
 follows
\begin{align}
&\frac{d\|\Delta  T\|^2_{L^2}}{dt}+\kappa \|\nabla ((-\Delta)T\|^2_{L^2} 
\leq C (\|\nabla \bv\|^2_{H^1}+\| \bv\|^2_{L^\infty})\| \Delta  T\|^2_{L^2}
\leq C\|\bv\|^2_{H^2}\| \Delta  T\|^2_{L^2}, \label{3EST_TH2}
\end{align}
where  we used  Among's inequality \eqref{agmon}  in the last inequality.
Therefore, applying \eqref{EST_opt_v} to \eqref{3EST_TH2} yields 
\begin{align}
& \sup_{t\in[0, t_f]} \|\Delta  T\|_{L^2}
\leq e^{C\int^{t_f}_0\|\bv\|^2_{H^2}\,dt}\| \Delta  T_0\|_{L^2}<\infty \label{4EST_TH2}
\end{align}
and
\begin{align}
&\kappa \int^{t_f}_0  \|\nabla ((-\Delta)T\|^2_{L^2} \,dt
\leq C\int^{t_f}_0\|\bv\|^2_{H^2}\| \Delta  T\|^2_{L^2}\,dt<\infty. \label{5EST_TH2}
\end{align}
This completes the proof.
\end{proof}

%
\noindent\textbf{Proof of Theorem \ref{SONC}.}
\begin{proof}
Let $h_i\in U_{\text{ad}}$  and $z_i=T'(\bv)\cdot h_i, i=1,2$. Then  we have 
\begin{equation}\label{z-equation_i}
	\begin{split}
		&\frac{\partial z_i}{\partial t}=\kappa \Delta z_i  - \bv \cdot \nabla z_i - h_i \cdot \nabla T,\quad \frac{\partial z_i}{\partial n}|_{\Gamma}=0,\\
		&z(x, 0)=0.
	\end{split}
\end{equation}
In light of  Corollary \ref{cor1}, we can  also obtain a higher regularity of  $z_i, i=1,2,$ than   \eqref{EST_zL2}.
To see this, taking the inner produce of \eqref{z-equation_i}  with $-\Delta z_i$ follows
 \begin{align}
		&\frac{1}{2}\frac{d \|\nabla z_i\|^2_{L^2}}{d t}+\kappa \|\Delta z_i\|^2_{L^2}\leq \|\bv\|_{L^\infty} \| \nabla z_i\|_{L^2}\|\Delta z_i\|_{L^2}
		+\|h_i\|_{L^4}\|\nabla T\|_{L^4} \|\Delta z_i\|_{L^2}\nonumber\\
		&\qquad\leq C \|\bv\|^2_{L^\infty} \| \nabla z_i\|^2_{L^2}
		+C\|\nabla h_i\|^2_{L^2}\|\Delta T\|^2_{L^2} +\frac{\kappa}{2}\|\Delta z_i\|_{L^2}. \label{0EST_zH1}
\end{align}
Thus
 \begin{align*}
		&\frac{d \|\nabla z_i\|^2_{L^2}}{d t}+\kappa \|\Delta z_i\|^2_{L^2}
		\leq C \|\bv\|^2_{L^\infty} \| \nabla z_i\|^2_{L^2}
		+C\|\nabla h_i\|^2_{L^2}\|\Delta T\|^2_{L^2},
\end{align*}
where by \eqref{4EST_TH2},
\[\int^{t_f}_0\|\nabla h_i\|^2_{L^2}\|\Delta T\|^2_{L^2}\,dt
\leq  \sup_{t\in[0, t_f]}\|\Delta T\|^2_{L^2} \int^{t_f}_0\|\nabla h_i\|^2_{L^2}\,dt \leq C(T_0, t_f) \|h_i\|^2_{U_{\text{ad}}}. \]
Consequently, 
 \begin{align}
		& \sup_{t\in[0, t_f]} \|\nabla z_i\|^2_{L^2}\leq 
		\int^{t_f}_0e^{C\int^{t_f}_{\tau} \|\bv\|^2_{L^\infty}\,ds}\|\nabla h_i\|^2_{L^2}\|\Delta T\|^2_{L^2} \, d\tau
		 \leq C(T_0, t_f) \|h_i\|^2_{U_{\text{ad}}} \label{EST_zH1}
\end{align}
and 
 \begin{align*}
		\kappa \int^{t_f}_0\|\Delta z_i\|^2_{L^2}\leq C \int^{t_f}_0  (\|\bv\|^2_{L^\infty} \| \nabla z_i\|^2_{L^2}
		+\|\nabla h_i\|^2_{L^2}\|\Delta T\|^2_{L^2})\, dt\leq   C(T_0, t_f) \|h\|^2_{U_{\text{ad}}}. 
\end{align*}

Next,  let  $Z=z'_1(\bv)\cdot h_2$. 
 Then $Z$ satisfies 
\begin{align}
	&\frac{\partial Z}{\partial t}=	\kappa \Delta Z -  h_2 \cdot \nabla z_1 -\bv \cdot \nabla Z-h_1 \cdot \nabla z_2,
	\quad Z|_{\Gamma}=0,\label{Z-equation}\\
		 &Z(x, 0)=0. \nonumber
\end{align}
Applying an $L^2$-estimate for $Z$ gives 
\begin{align*}
&\frac{1}{2}\frac{d\| Z\|^2_{L^2}}{dt}+\kappa \|\nabla Z\|^2_{L^2}
\leq\|\nabla h_2\|_{L^2} \|\nabla z_1\|_{L^2}\|\nabla Z\|_{L^2}
+\|\nabla h_1\|_{L^2}\|\nabla z_2\|_{L^2}\|\nabla Z\|_{L^2}\\
&\qquad\leq \|\nabla h_2\|^2_{L^2} \| \nabla z_1\|^2_{L^2}+\frac{\kappa}{4}\|\nabla Z\|^2_{L^2}
+\|\nabla h_1\|^2_{L^2} \| \nabla z_2\|^2_{L^2}+\frac{\kappa}{4}\|\nabla Z\|^2_{L^2},
\end{align*}
which,  together with \eqref{EST_zH1}, follows
\begin{align*}
&\frac{d\| Z\|^2_{L^2}}{dt}+\kappa \|\nabla Z\|^2_{L^2}
\leq C(\|\nabla h_2\|^2_{L^2} \| \nabla z_1\|^2_{L^2}+\|\nabla h_1\|^2_{L^2} \| \nabla z_2\|^2_{L^2} )\\
&
\qquad\leq C(T_0, t_f)(  \|\nabla h_2\|^2_{L^2} \|h_1\|^2_{U_{\text{ad}}}
+ \|\nabla h_1\|^2_{L^2} \|h_2\|^2_{U_{\text{ad}}}).
\end{align*}
 Therefore, 
\begin{align}
&\| Z\|^2_{L^2}+\kappa\int^{t}_0\|\nabla Z\|^2_{L^2}\,dt 
 \leq C(T_0, t_f) \| h_1\|^2_{U_{\text{ad}}} \|h_2\|^2_{U_{\text{ad}}}, \quad t\in[0,t_f].
 \label{EST_ZH1}
\end{align}
By Lemma \ref{lem0}, \eqref{EST_zL2} and \eqref{EST_ZH1},  it can be easily verified   that  the terms on the right hand side of \eqref{Z-equation} are all in $L^1(0, t_f; (H^{1}(\Omega))')$, and hence  $ \frac{\partial Z}{\partial t} \in L^1(0, t_f; (H^{1}(\Omega))')$. Thus there exists a unique solution to \eqref{Z-equation}, which implies that $T(\bv)$ is twice G$\hat{a}$teaux differentiable at $\bv\in U_{\text{ad}}$ satisfying the optimality condition \eqref{opt_cond}, with respect to $h_1$ and $h_2$, so is $J(\bv)$.

Now differentiating $J'(\bv)\cdot h_1$ once again in the direction $h_2\in U_{\text{ad}}$ gives
\begin{align}
		J''(\bv) \cdot (h_1,h_2) =&\alpha (D^*Dz_2(t_f), z_1(t_f)) + \alpha (D^*DT(t_f), Z(t_f))
		+\beta \int^{t_f}_0(D^*Dz_2, z_1)\,dt\nonumber\\
		&+ \beta \int^{t_f}_{0} (D^*DT, Z)\,dt
		+ \gamma \int^{t_f}_0(Ah_2, h_1)\,dt.\label{J_2diff}
\end{align}
Next  taking  the  inner product of \eqref{Z-equation} with $q$ and applying 
 \eqref{2EST_tri}, we get
\begin{align*}
		&	\alpha (D^*DT(t_f), Z(t_f))-\int^{t_f}_0(Z, \frac{\partial q}{\partial t})\,dt=\kappa \int^{t_f}_0(Z, \Delta q)\,dt
		+\int^{t_f}_0( z_1, h_2 \cdot \nabla q) \,dt\\
		&\qquad+\int^{t_f}_0 ( Z, \bv\cdot \nabla q)\,dt +\int^{t_f}_0( z_2, h_1\cdot \nabla q)\,dt.
\end{align*}
With the help of   the adjoint equations  \eqref{opt_q},  we obtain   
\begin{align*}
		\alpha (D^*DT(t_f), Z(t_f))+\beta\int^{t_f}_0(Z, D^*DT)\,dt=\int^{t_f}_0( z_1, h_2 \cdot \nabla q)\,dt+\int^{t_f}_0( z_2, h_1\cdot \nabla q)\,dt.
\end{align*}
Therefore, \eqref{J_2diff} becomes
\begin{align*}
		J''(\bv) \cdot (h_1,h_2) =&	\alpha (D^*Dz_2(t_f), z_1(t_f))  +\beta \int^{t_f}_0(D^*Dz_2, z_1)\,dt
		+\int^{t_f}_0( z_1, h_2 \cdot \nabla q)\,dt\nonumber\\
		&+\int^{t_f}_0( z_2, h_1\cdot \nabla q))\,dt
                 + \gamma\int^{t_f}_0(Ah_2, h_1)\,dt.
\end{align*} 
Setting $h_1=h_2=h$ and $z_1=z_2=z=T'(\bv)\cdot h$ follows
\begin{align}\label{optimality_4}
		J''(\bv) \cdot (h, h) =\alpha \|Dz(t_f)\|^2_{L^2}+\beta \int^{t_f}_0\|Dz\|^2_{L^2}\,dt+2\int^{t_f}_0( z, h\cdot \nabla q)\,dt
		+\gamma\int^{t_f}_0\|A^{1/2}h\|^2_{L^2}\,dt.
\end{align} 
Furthermore, by \eqref{1EST_tri}, \eqref{EST_zL2}, \eqref{EST_qL2} and \eqref{EST_zH1}, we get
\begin{align*}
 \|Dz(t_f)\|^2_{L^2} \leq    \frac{C}{\kappa} \|T_0\|^2_{L^\infty} \| h\|^2_{U_{\text{ad}}},  
\end{align*}
\begin{align*}
		\int^{t_f}_0\|Dz\|^2_{L^2}\,dt \le C \int^{t_f}_0\|\nabla z\|^2_{L^2}\,dt 
		 \le   \frac{C}{\kappa^2} \|T_0\|^2_{L^\infty} \| h\|^2_{U_{\text{ad}}}, 
\end{align*} 
and
\begin{align*}
&\left|\int^{t_f}_0( z, h\cdot \nabla q)\,dt\right| 
\leq  C\int^{t_f}_0 \|\nabla z\|_{L^2}\|\nabla h\|_{L^2}\|\nabla q\|_{L^2}\,dt\nonumber\\
&\qquad\leq    C\sup_{t\in[0, t_f]}\|\nabla z\|_{L^2}(\int^{t_f}_0 \|\nabla h\|^2_{L^2})^{1/2} (\int^{t_f}_0 \|\nabla q\|^2_{L^2}\,dt)^{1/2}
\leq  C(T_0, t_f) \|h\|^2_{U_{\text{ad}}}.
\end{align*}
 As a result, 
\begin{align*}
		|J''(\bv) \cdot (h, h)| &\leq  C(T_0, t_f)\left(   \frac{\alpha}{\kappa}+ \frac{\beta}{\kappa^2}\right) \|T_0\|^2_{L^\infty} \| h\|^2_{U_{\text{ad}}}
		+\gamma\| h\|^2_{U_{\text{ad}}}=( C(T_0, t_f, \kappa,  \alpha, \beta)+\gamma)\| h\|^2_{U_{\text{ad}}}
\end{align*} 
and 
\begin{align*}
	J''(\bv) \cdot (h, h) \geq 
	&-2\int^{t_f}_0( z, h\cdot \nabla q)\,dt
		+\gamma\int^{t_f}_0\|A^{1/2}h\|^2_{L^2}\,dt
		= (\gamma-C(T_0, t_f, \kappa,  \alpha, \beta))\| h\|^2_{U_{\text{ad}}}.
	\end{align*}  
Therefore, letting $\gamma$  large enough such that  
\begin{align}
\gamma-C(T_0, t_f, \kappa,  \alpha, \beta)
  \geq c_0>0, \label{cond_gamma}
\end{align}
 we obtain   \eqref{opt_2rd}.
  \end{proof}

%
%
%

\begin{remark}\label{prop_aveT}
For $T_0\in L^2(\Omega)$ and $\bv\in L^2(0, \infty;, H)$,  $\|DT\|_{L^2}$ obeys an exponential decay rate  in time.  
\end{remark}

\begin{proof}
Taking the inner product of \eqref{sta_T} with $D^*DT$ and applying Greens' formula and \eqref{2EST_tri}, we have
  \begin{align}
\frac{1}{2}	\frac{d \|DT\|^2_{L^2}}{dt}&=\kappa( \Delta T, D^*DT)-(\bv \cdot \nabla T, D^*DT)\nonumber\\
&=\kappa\langle \frac{\partial  T}{\partial n}, D^*DT\rangle \rangle_{\Gamma}
 -\kappa (\nabla T, \nabla (D^*DT) )+(\bv  T, \nabla (D^*DT)).\label{EST_DT}
\end{align}
Since  $\langle T\rangle$ is a  function of $t$ and $D^*D=D$, we have 
$\nabla (D^*DT) =\nabla (DT) =\nabla (T-\langle T\rangle)=\nabla T,$
and  hence  using \eqref{2EST_tri} and Stokes formula follows
$$(\bv \cdot \nabla T, D^*DT)=-(T, \bv \cdot \nabla  (D^*DT))=-(\bv, T\nabla T)=-\frac{1}{2}(\bv, \nabla (T^2))=0.$$
 Therefore, \eqref{EST_DT} becomes 
  \begin{align}
&\frac{1}{2}	\frac{d \|DT\|^2_{L^2}}{dt}+\kappa\|  \nabla (DT)\|^2_{L^2}=0.\label{2EST_DT}
\end{align}
Further applying   Gr\"{o}nwall's  inequality and Poncar\'{e} inequality  
we derive that 
   \begin{align}
\| DT\|^2_{L^2}\leq e^{-C\kappa t}\| DT_0\|^2_{L^2},
\label{EST_TH2}
\end{align}
which establishes the claim. 
\end{proof}

\end{document}